\documentclass[12pt,A4]{article}
\usepackage[french,english]{babel}
\usepackage{amsmath,amsfonts,amssymb,amsthm,mathrsfs,bbm}
\usepackage[font=sf, labelfont={sf,bf}, margin=1cm]{caption}
\usepackage{graphicx,graphics}
\usepackage{epsfig}
\usepackage{latexsym}
\usepackage[applemac]{inputenc}
\usepackage{ae,aecompl}
\usepackage{pstricks}
\usepackage{enumerate}
\usepackage{xcolor}

\usepackage[pdfpagemode=UseNone,bookmarksopen=false,colorlinks=true,urlcolor=blue,citecolor=blue,citebordercolor=blue,linkcolor=blue]{hyperref}

\usepackage[normalem]{ulem}
\usepackage[top=2.5cm,left=1.5cm,right=1.5cm,bottom=2.5cm]{geometry}
\usepackage{comment}
\usepackage{eulervm,palatino}

\newcommand\Es[1]{\mathbb{E}\left[#1\right]}
\renewcommand\Pr[1]{\mathbb{P}\left(#1\right)}

\def \R {\mathbb R}

\def \L {\mathbb{L}}
\def \E {\mathcal E}
\def \oE {\overline{\mathcal{E}}}
\def \P {\mathbb{P}}
\def \Exp {\textnormal{\textsf{Exp}}}
\newtheorem{thm}{Theorem}
\newtheorem{lem}{Lemma}[section]
\newtheorem{prop}[lem]{Proposition}
\newtheorem{cor}[lem]{Corollary}
\newtheorem{thmbis}[lem]{Theorem}

\theoremstyle{definition}

\newtheorem{defn}[lem]{Definition}

\title{  \vspace {-2cm}\textbf{A predator-prey SIR type dynamics on large complete graphs with three phase transitions}}
\date{}
\author{Igor Kortchemski\thanks{DMA, École Normale Supérieure, Paris; E-mail: igor.kortchemski@normalesup.org}}

\DeclareSymbolFont{extraup}{U}{zavm}{m}{n}
\DeclareMathSymbol{\varheart}{\mathalpha}{extraup}{86}
\DeclareMathSymbol{\vardiamond}{\mathalpha}{extraup}{87}

\makeatletter
\renewcommand*{\@fnsymbol}[1]{\ensuremath{\ifcase#1\or  \vardiamond \or \clubsuit\or \spadesuit\or
   \mathsection\or \mathparagraph\or \|\or **\or \dagger\dagger
   \or \ddagger\ddagger \else\@ctrerr\fi}}
\makeatother

\begin{document}
\maketitle

\let\thefootnote\relax\footnotetext{ \\\emph{MSC2010 subject classifications}. Primary 60J80, 60F05, 60F25; secondary 92D30. }
 
\vspace {-0.5cm}

\begin{abstract}  
We study a variation of the SIR (Susceptible/Infected/Recovered)  dynamics on the complete graph, in which infected individuals may only spread to neighboring susceptible individuals at fixed rate $ \lambda>0$ while recovered individuals may only spread to neighboring infected individuals at fixed rate $1$. This is also a variant of the so-called chase-escape process introduced by Kordzakhia and then Bordenave. Our work is the first study of this dynamics on complete graphs. Starting with one infected and one recovered individuals on the complete graph with $N+2$ vertices, and stopping the process when one type of individuals disappears, we study the asymptotic behavior of the probability that the infection spreads to the whole graph as $N \rightarrow \infty$ and show that for $ \lambda \in (0,1)$ (resp. for $ \lambda>1$), the infection dies out (resp. does not die out) with probability tending to one as $N \rightarrow \infty$, and that the probability that the infection dies out tends to $1/2$ for $ \lambda=1$.  We also establish limit theorems concerning the final state of the system in all regimes and show that two additional phase transitions occur in the subcritical phase $ \lambda \in (0,1)$: at $ \lambda=1/2$ the behavior of the expected number of remaining infected individuals changes, while at $ \lambda=( \sqrt {5}-1)/2$ the behavior of the expected number of remaining recovered individuals changes.  We also study the outbreak sizes of the infection, and show that the outbreak sizes are small (or self-limiting) if $ \lambda  \in (0,1/2]$, exhibit a power-law behavior for $ 1/2 < \lambda <1$, and are pandemic for $ \lambda \geq 1$. Our method relies on different couplings: we first couple the dynamics with  two independent Yule processes by using an Athreya--Karlin embedding, and then we couple the Yule processes with Poisson processes thanks to Kendall's representation of Yule processes.

\end{abstract}

\section*{Introduction}

We investigate the asymptotic behavior of a variation of the SIR (Susceptible/Infected/Recovered) dynamics on the complete graph. It is a stochastic dynamics of two competing species and may be informally defined as follows. Vertices can only be of three types: susceptible, infected or recovered. At fixed rate $ \lambda>0$, infected individuals may only spread their infection to neighbors who are susceptible (when an infected individual spreads its infection to a neighboring susceptible individual, both of them are then infected), while at fixed rate $1$ recovered individuals may only spread their recovered state to infected neighbors (see below for a formal definition). This dynamics is called the \emph{chase-escape} process and appears in a work by Bordenave  \cite{Bor14}, who analyzed its behavior on infinite trees (see also \cite{Kor13}). 

  Earlier, Kordzakhia \cite{Kord05} has introduced and studied the \emph{escape} process on infinite trees, in which infected individuals still may only spread their infection to neighboring susceptible individuals at fixed rate $ \lambda>0$, but recovered individuals may  spread at fixed rate $1$ to neighbors that are either infected or susceptible individuals. Later, Bordenave \cite{Bor08} introduced the \emph{rumor-scotching process}, in which infected individuals may only spread to neighboring susceptible individuals at fixed rate $ \lambda>0$, and recovered individuals may spread at fixed rate $1$ to infected neighbors but only through edges that have already spread an infection: See \cite{Bor08} for a formal definition, where infected individuals are individuals propagating a rumor while recovered individuals are struggling to scotch it (see in particular \cite {CS14}, where the authors propose to use this model to predict the activity of Facebook). One may also think of the susceptible individuals as vacant vertices, infected individuals as prey, and recovered individuals as predators, and view the chase-escape process as a random foodchain. The rumor-scotching process can be seen a directed version of the chase-escape process.  Under suitable initial conditions, the behavior of the escape, chase-escape and rumor-scotching processes is the same on infinite trees. In \cite{Bor08}, Bordenave studied the rumor-scotching process on the complete graph, and showed that its scaling limit is the birth-and-assassination process, which was introduced by Aldous \& Krebs \cite{AK90}.  We finally mention that the chase-escape process is a variant of the famous and extensively studied Daley--Kendall model \cite {DK65} for rumor propagation (in which, in addition, an Infected individual may become recovered if it enters in contact with another infected individual or, in other words, when two individuals spreading the rumor meet, one of them stops spreading it).

In this work, we are interested in the asymptotic behavior of the chase-escape process on the complete graph with $N+2$ vertices as $N \rightarrow \infty$, starting with one infected and one recovered individual and stopped when either no susceptible or no infected individuals remain. To our knowledge, this is the first study of the chase-escape process on graphs different from trees. We will establish that the probability that no susceptible individual remains tends to $0$ as $N \rightarrow \infty$ if and only if $ \lambda > 1$. We also establish limit theorems concerning the state of the system at its absorbing state in all (subcritical, critical and supercritical) regimes. 

\paragraph{The chase-escape process.} We now give a formal definition of the chase-escape process following  Bordenave \cite{Bor14}. Let $G=(V,E)$ be a locally finite connected graph.  Set $ \mathfrak {X}= \{ S,I,R\}^{V}$ and for every $v \in V$, let $I_{v}, R_{v}:  \mathfrak {X}\rightarrow  \mathfrak {X}$ be the two maps defined by $(I_{v}(x))_{u}=(R_{v}(x))_{u}=x_{u}$ if $u \neq v$ and $(I_{v}(x) )_{v}=I$ and $(R_{v}(x))_{v}=R$ with $x=(x_{u})_{u \in V}$. By definition, the chase-escape process of infection intensity $ \lambda>0$ is the Markov process taking values in $ \mathfrak {X}$ with transition rates
$$Q(x,I_{v}(x))= \lambda  \cdot \mathbbm {1}_{ \{ x_{v}=S\} }  \cdot  \sum_{ \{ u,v\} \in E } \mathbbm {1}_{ \{ x_{u}=I\} }, \quad Q(x,R_{v}(x))= \mathbbm {1}_{ \{ x_{v}=I\} }  \cdot  \sum_{\{ u,v\} \in E} \mathbbm {1}_{ \{ x_{u}=R\} }  \qquad (v \in V, x \in \mathfrak {X}).$$  
This means that each infected vertex infects its susceptible neighbors at rate $ \lambda$ and that each recovered individual spreads its recovered state to its infected neighbors at rate $1$. In particular, an infected individual who recovers then stays recovered indefinitely. Note that this dynamics differs from the classical SIR epidemic model, where infected individuals recover at a fixed rate (not depending on their neighborhood). 

In the following, we will always consider the chase-escape process on the complete finite graph $K_{N+2}$ on $N+2$ vertices, where $N \geq 1$ is an integer, starting with one infected individual, one recovered individual and the other $N$ individuals all being susceptible. 
In addition,
\begin {center} \emph{we stop the chase-escape process once either no infected or no susceptible individuals remain.} \end {center}
Hence the absorbing states of this process are the states where no susceptible individuals are present (which we interpret as the fact that the infection has spread to the entire population) and where there are no infected individuals but where susceptible individuals remain (which we interpret as the fact that the infection has died out). The motivation for stopping the process when no susceptible individuals remain is to try understand the severity of the infection: the infection is severe if at the time of absorption there are a lot of infected individuals and a few recovered ones, while the infection is less severe if at the time of absorption there are a few infected individuals and many recovered ones.

\paragraph {The critical value of $ \lambda$.}  Let  $ E^{N}_{ext}$ be the event that there exists a (random) time when no susceptible individuals remain. Hence $\P(E^{N}_{ext})$ is the probability that the infection spreads to the whole graph. Denote also by $ {}^{c} E^{N}_{ext}$ the complementary event where the infection dies out. We first identify
 $ \lambda=1$ as a critical value, and also give the limiting value of $\P(E^{N}_{ext})$ as $N \rightarrow \infty$ for $ \lambda=1$.

\begin{thm} \label {thm:critical}We have:
$$ \P(E^{N}_{ext})  \quad\mathop{\longrightarrow}_{N \rightarrow \infty} \quad  \begin {cases} \quad 0 & \textrm{ if } \quad  \lambda \in (0,1)\\
 \quad \frac{1}{2} & \textrm{ if } \quad \lambda=1 \\
\quad 1 & \textrm{ if } \quad  \lambda>1. \end{cases}$$
\end{thm}
Here, note that \emph{ext} refers to the extinction of susceptible individuals. 

\paragraph{The number of remaining susceptible individuals.} Denote by $S^{(N)}$ the number of remaining susceptible individuals once the chase-escape process has reached an absorbing state. Our next contribution describes the asymptotic behavior of $S^{(N)}$ as $N \rightarrow \infty$.  Here and later, $ \Exp( \lambda)$ denotes an exponential random variable of parameter $ \lambda>0$, independent of all the other mentioned random variables (in particular, different occurrences of $\Exp(1)$ denote different independent random variables). The symbol $\displaystyle \mathop{\longrightarrow}^{(d)}$ denotes convergence in distribution for random variables.

\begin{thm}\label{thm:stateS1} 
\begin{enumerate}
 \item[(i)] Assume that $ \lambda \in (0,1)$.  Then  \begin{equation}
 \label{eq:S1}\displaystyle \frac{S^{(N)}}{N^{1- \lambda}}  \quad\mathop{\longrightarrow}^{(d)}_{N \rightarrow \infty} \quad  \textnormal{\textsf{Exp}}(1)^{ \lambda}.
 \end{equation}

 \item[(ii)] Assume that $ \lambda=1$.  Then for every $i \geq 0$, $ \displaystyle \Pr{S^{(N)}=i}  \rightarrow {1}/{2^{i+1}}$ as $N \rightarrow \infty$.
  In other words, $S^{(N)}$ converges in distribution to the random variable $G$ such that $ \P (G=i)=1/2^{i+1}$ for $i \geq 0$.
  \item[(iii)] Assume that $ \lambda>1$.  Then $S^{(N)}$ converges in probability to $0$ as $ N \rightarrow \infty$.  \end{enumerate}
\end {thm}

\noindent We make several comments on these results:
\begin{enumerate}
 \item[(i)] It is interesting to observe that for $ \lambda=1$, $S^{(N)}$ converges in distribution as $ N \rightarrow \infty$ without scaling. 
 \item[(ii)] When $ \lambda=1$, since $S^{(N)}=0$ on the event $E^{N}_{ext}$ and $ \Pr {E^{N}_{ext}} \rightarrow1/2$, assertion (ii) may be reformulated by saying that conditionally on ${}^{c}E^{N}_{ext}$, $ S^{(N)}$ converges in distribution to the positive random variable $G'$ whose law is given by $ \P(G'=i)=1/2^{i}$ for $i \geq 1$. Hence, for $ \lambda=1$, the number of non-infected individuals in a population which is not entirely infected is asymptotically constant (meaning that $ \lim_{N \rightarrow \infty} \Pr {S^{(N)} \geq K}$ tends to $0$ as $K \rightarrow \infty$).
 \end {enumerate}

\paragraph{The number of remaining recovered individuals.}  Denote by $R^{(N)}$ the number of recovered individuals once the chase-escape process has reached an absorbing state. Recall that chase-escape process is stopped once either no infected or no susceptible individuals remain.  We are next interested in the asymptotic behavior of $R^{(N)}$ as $N \rightarrow \infty$.

\begin{thm}\label{thm:stateR1} The following assertions hold.\begin{enumerate}
\item[(i)] Assume that $ \lambda \in (0,1)$. Then 
\begin{equation}
\label{eq:R1} \frac{N-R^{(N)}}{N^{1- \lambda}}  \quad\mathop{\longrightarrow}^{(d)}_{N \rightarrow \infty} \quad \Exp(1) ^{ \lambda}.
\end{equation}
\item[(ii)] Assume that $ \lambda=1$.  Then  \begin{equation}
\label{eq:R12}\frac{R^{(N)}}{N}  \quad\mathop{\longrightarrow}^{(d)}_{N \rightarrow \infty} \quad \frac{1}{2} \delta_{1} +  \frac{1}{(1+x)^{2}} \mathbbm {1}_{[0,1]}(x) dx,
\end{equation}
where $ \delta_{1}$ is a Dirac measure at $1$ and ${(1+x)^{-2}} \mathbbm {1}_{[0,1]}(x) dx$ denotes the measure with density ${(1+x)^{-2}}$ on $[0,1]$.
\item[(iii)] Assume that $ \lambda>1$. Then:
 \begin{equation}
 \label{eq:cvd1} \frac{R^{(N)}}{N^{1/\lambda}}  \quad\mathop{\longrightarrow}^{(d)}_{N \rightarrow \infty} \quad  \textnormal{\textsf{Exp}}(\textnormal{\textsf{Exp}}(1)^{1/\lambda}),
 \end{equation}
 where $\textnormal{\textsf{Exp}} (\textnormal{\textsf{Exp}}(1)^{1/\lambda})$ is an exponential random variable of independent random parameter $\textnormal{\textsf{Exp}}(1)^{1/\lambda}$.
\end {enumerate}
\end {thm}
\noindent We make several comments on these results:

\begin{enumerate}
 \item[(i)]  When $ \lambda=1$, conditionally on ${}^{c}E^{N}_{ext}$, we have $R^{(N)}+S^{(N)}=N+2$. Hence, by remark (ii) following Theorem \ref {thm:stateS1}, conditionally on ${}^{c}E^{N}_{ext}$,  $R^{(N)}/N$ converges in probability to $1$. In particular, assertion (ii) of Theorem \ref {thm:stateR1} may be reformulated by saying that conditionally on $E^{N}_{ext}$, the law of $ R^{(N)}/N$ converges in distribution to the measure $ {2}{(1+x)^{-2}} \mathbbm {1}_{[0,1]}(x) dx$ as $ N \rightarrow \infty$.  
 \item[(ii)]It is interesting to note the heavy-tail behavior of the limiting random variable appearing in (iii); its moments are given by $ \Es {\textnormal{\textsf{Exp}}(\textnormal{\textsf{Exp}}(1)^{1/\lambda}) ^{s}}= { \Gamma(1+s)}/{  \Gamma(1-s/ \lambda)}$ for $-1 < s < \lambda$, and its tail by
$  \Pr {\textnormal{\textsf{Exp}}(\textnormal{\textsf{Exp}}(1)^{1/\lambda})> u }  \sim  { \Gamma( \lambda+1)} \cdot {u^{ -\lambda}}$ as $u \rightarrow \infty$.
 \item[(iii)] Note also the dissymmetry in the asymptotic behavior of the system for $ \lambda=1$: when the infection does not spread to the whole graph, only a few susceptible individuals remain and almost every individual is recovered, while when the infection spreads to the whole graph, the number of recovered individuals is a random proportion of the total number of individuals. 
\end{enumerate}

 \paragraph{The number of remaining infected individuals: Outbreak sizes.}   Denote by $I^{(N)}$ the number of infected individuals once the chase-escape process has reached an absorbing state. We are finally interested in the asymptotic behavior of $I^{(N)}$.

 \begin{thm}\label{thm:stateI1} The following assertions hold.\begin{enumerate}
\item[(i)] If $ \lambda \in (0,1)$, then $I^{(N)}$ converges in probability to $0$ as $N \rightarrow \infty$. 
\item[(ii)] If $ \lambda=1$, then 
\begin{equation}
\label{eq:I1}\frac{I^{(N)}}{N}  \quad\mathop{\longrightarrow}^{(d)}_{N \rightarrow \infty} \quad   \frac{1}{(2-x)^{2}} \mathbbm {1}_{[0,1]}(x) dx.
\end{equation}
\item[(iii)] If $ \lambda>1$, 
\begin{equation}
\label{eq:I2}\frac{N-I^{(N)}}{N^{1/ \lambda}}  \quad\mathop{\longrightarrow}^{(d)}_{N \rightarrow \infty} \quad  \textnormal{\textsf{Exp}}(\textnormal{\textsf{Exp}}(1)^{1/\lambda}).
\end{equation}
 
\end {enumerate}
\end {thm}

\paragraph {Outbreak sizes.} In the language of epidemiology, an outbreak is defined to be small, or self-limiting, if the total average number of infected individuals at the absorption time does not scale with $N$, and an outbreak is pandemic if the total average number of infected individuals at the absorption time is a positive fraction of the population. In order to determine the nature of the outbreak sizes in the chase-escape process, we investigate if the convergences appearing in Theorems \ref {thm:stateS1}, \ref {thm:stateR1} and \ref {thm:stateI1} hold in $ \L^{p}$. Our results (which we state below in Sec.~\ref {sec:LP} in order to shorten the Introduction) establish in particular that the outbreak sizes are small if $ \lambda  \in (0,1/2]$ (they even tend to $0$ for $  \lambda \in (0,1/2)$ and tend to $1$ for $ \lambda=1/2$), exhibit a power-law behavior for $ 1/2 < \lambda <1$, and are pandemic for $ \lambda \geq 1$ (however, for $ \lambda=1$,  the mean fraction of non-infected individuals tends to a positive number, while for $ \lambda>1$, the mean fraction of not-infected individuals  tends to $0$).

 This contrasts heavily with the outbreak sizes in the SIR model on the complete graph, where outbreaks are small in the subcritical case, pandemic in the supercritical case and exhibit a power-law behavior in the critical case (see e.g. \cite {BK04} for a study of the average outbreak sizes in the SIR model).

\paragraph{Phase transitions.} We have seen in Theorem \ref{thm:critical} that $ \lambda=1$ is the critical value for the extinction of susceptible individuals in the large population limit. Theorems \ref {thm:stateS2}, \ref {thm:stateR2} and \ref {thm:stateI2} (stated in Sec.~\ref {sec:LP} below) also show that the chase-escape process on large complete graphs exhibits  two additional phase transitions  in the subcritical phase $ \lambda \in (0,1)$. The first one is at $ \lambda =1/2$. Below this value, the expected final number of infected individuals $ \Es{I^{(N)}}$ tends to $0$ as the size of the graph grows. At the value $ \lambda=1/2$, $ \Es{I^{(N)}}$ tends to $1$ in the large population limit.  For $ \lambda \in (1/2,1)$, $ \Es{I^{(N)}}$ grows as a power of $N$ with a positive exponent less than $1$.

Note also that at $ \lambda=1$ the behavior of the growth of $ \Es{I^{(N)}}$ also changes: for $ \lambda=1$, $ \Es{I^{(N)}}/N$ converges to a positive constant which is less than $1$, while for $ \lambda>1$, $ \Es{I^{(N)}}$ is asymptotic to $N$.

Another phase transition in the subcritical phase $ \lambda \in (0,1)$ occurs at $ \lambda=(1- \sqrt{5})/2$ when looking at the second order of the asymptotic behavior of $ \Es{R^{(N)}}$: for $0< \lambda \leq (\sqrt {5}-1)/2$, $N- \Es{R^{(N)}}$ is of order $ N^{1- \lambda}$, while for $ (\sqrt {5}-1)/2 \leq  \lambda <1$, $N- \Es{R^{(N)}}$ is of order $ N^{2- 1/\lambda}$. This is explained by the fact that for $0 < \lambda <  (\sqrt {5}-1)/2$ the main contribution to $N- \Es{R^{(N)}}$ comes from the remaining susceptible individuals, while for $(\sqrt {5}-1)/2 <  \lambda <1$ the main contribution to $N- \Es{R^{(N)}}$ comes from the remaining infected individuals. For $ \lambda=(\sqrt {5}-1)/2$, these contributions are of the same order.

\paragraph{Related models.} Let us mention several other similar but different models which have appeared in the literature. If no recovered individuals are present in the beginning, this dynamics is the so-called Richardson's model \cite {Ric73}. Also, H\"aggstr\"om \& Pemantle  \cite {HP98} and Kordzakhia  \& Lalley \cite {KL05} have studied an  extension of Richardson's model with two species, in which infected and recovered individuals may only spread to susceptible individuals.

\paragraph{Techniques.}  The main idea is to couple the dynamics with two independent Yule processes by using an Athreya--Karlin embedding \cite{AK68}. Recall that in a Yule process of parameter $ \lambda$, each individual lives a random independent time distributed as an exponential random variable of parameter $ \lambda$ and produces two offspring at its death. Let $ (\mathcal{Y}_{t})_{t \geq 0}$ be a Yule process of parameter $ \lambda$ starting with one individual and let $ (\mathcal{Z}_{t})_{t \geq 0}$ be a Yule process of parameter $1$ starting with one individual.  Let $\overline{\mathcal{Y}}$ be the process $ \mathcal{Y}$ time-reversed at its $N$-th jump. More precisely, if $ t_{N}$ denotes the $N$-th jump of $ \mathcal{Y}$, then $ \overline{\mathcal{Y}}_{t}= \mathcal{Y}_{(t_{N}-t)-}$ for $0 \leq t \leq t_N$. Then, informally, the chase-escape process can be constructed in such a way that the discontinuities of $ \overline{\mathcal{Y}}$ are the times when a susceptible vertex is infected and the discontinuities of $ \mathcal{Z}$ are the times when a recovered vertex spreads to an infected one. In particular, $ \overline{\mathcal{Y}}_{t}$ represents the number of susceptible vertices at time $t$ and $ \mathcal{Z}_{t}$ represents the number of recovered vertices at time $t$. The absorption time is then the first time $t$ when either $ \overline{\mathcal{Y}}_{t}=0$ or $ \overline{\mathcal{Y}}_{t}+ \mathcal{Z}_{t}=N+2$.   See Theorem \ref {thm:coupling} for a precise statement. A useful feature of this coupling is that the same process $ \overline{\mathcal{Y}}$ is used for different values of $N$. 

In particular, the dynamics can be viewed as a generalized non-conservative P\'olya urn process with two urns (see e.g.  \cite {Jan04,LP07} for a study of  P\'olya urns using branching processes). The first urn has $N$ balls in the beginning (and represents susceptible individuals) and the second one has $1$ ball in the beginning (and represents recovered individuals). Balls are selected uniformly at random, with an activity (or weight) $ \lambda$ for balls of the first bin, and activity $1$ for balls of the second bin. When a ball from the first bin is selected, it is removed. However, when a ball from the second bin is selected, it is replaced in the second bin, and one additional ball is added to the second bin. The process is stopped at the first time when either the first bin is empty or the two bins contain $N+2$ balls together. Due to this very particular stopping time, we may not apply known results on P\'olya urns in our case.

In order to analyze the chase-escape process with this coupling, we use Kendall's representation of Yule processes, which states that if $ (\mathcal{P}_{t})_{t \geq 0}$ is a Poisson process of parameter $ \lambda$ and $ \mathcal{E}$ an independent exponential random variable of parameter $1$, then $( \mathcal{P}_{ \mathcal{E}(e^{t}-1) } +1)_{t \geq 0}$ is a Yule process of parameter $ \lambda$ with terminal value $ \mathcal{E}$ (see Section \ref {sec:Yule} for details and the definition of the terminal value).  This device allows us to transfer calculations on Yule processes to more tractable calculations on Poisson processes.

Finally, we mention that since $S^{(N)}+I^{(N)}+R^{(N)}=N+2$, it is sufficient to establish one of the two Theorems \ref{thm:stateS1},  \ref{thm:stateR1}  and  \ref{thm:stateI1} \ref{thm:stateI2} to get the other one.

\paragraph{Acknowledgments. } I am deeply indebted to Itai Benjamini for suggesting me to study this problem as well as to the Weizmann Institute of Science for hospitality, where this work begun. I would also like to thank Pascal Maillard for stimulating discussions, and an anonymous referee for a very careful reading as well as for many useful remarks that improved the paper.

\section{Decoupling in continuous time and Yule processes}

\subsection{A pure birth and a pure death chain}

Recall from the Introduction that the chase-escape process is run on the complete graph with $N+2$ vertices, where $N \geq 1$ is an integer, starting with one infected individual, one recovered individual and the other $N$ individuals being susceptible. The process is stopped at the first time $T$ when either no infected or no susceptible individuals remain. For every $t \geq 0$, let $S(t)$, $I(t)$ and resp. $R(t)$ be the number of susceptible, infected and resp. recovered individuals at time $t$ (we forget the dependence in $N$ to simplify notation). In particular, note that
$$T= \inf \{  t \geq 0; \quad S(t)=0 \textrm { or } S(t)+R(t)=N+2\}$$
and notice that $S(t)+I(t)+R(t)=N+2$ for $0 \leq t \leq T$. 
Let $(T_{n})_{1 \leq n \leq  \zeta }$ be the increasing sequence of discontinuity times of the process $(S(t),R(t))_{0 \leq t \leq T}$ and set $T_{0}=0$ by convention.

We now introduce a two-type Markov branching process $( \mathcal{S}_{N}(t), \mathcal{R}(t))_{t \geq 0}$ such that $( \mathcal{S}_{N}(t))_{t \geq 0}$ is a pure death chain and $( \mathcal{R}(t))_{t \geq 0}$ is an independent pure birth chain. More precisely, the chain $ \mathcal{S}_{N}$ starts with $N$ individuals (here we keep the subscript $N$ to emphasize the dependence in $N$). Each individual dies after an exponential time of parameter $ \lambda$, all independently. Finally, the independent chain $ \mathcal{R}$ starts with one individual and each individual gives birth after an exponential time of parameter $ 1$, all independently. Set
$$ \mathcal{T}= \inf \{ t \geq 0;  \mathcal{S}_{N}(t)=0 \textrm { or } \mathcal{R}(t) + \mathcal{S}_{N}(t)=N+2   \}.$$
and let $(\mathcal{T}_{i})_{1 \leq i \leq  \xi}$  be the increasing sequence of discontinuity times of the process $( \mathcal{S}_{N}(t), \mathcal{R}(t))_{ 0 \leq t \leq \mathcal{T} }$ and set $ \mathcal{T}_{0}=0$ by convention.

\begin {thmbis}\label {thm:coupling} For every integer $N \geq 1$, we have:
$$ (S(T_i), R(T_i))_{0 \leq i \leq  \zeta}  \quad\mathop{=}^{(d)} \quad ( \mathcal{S}_{N} ( \mathcal{T} _{i}),  \mathcal{R} ( \mathcal{T} _{i}))_{0 \leq i \leq  \xi}.$$
\end {thmbis}

\begin {proof} Fix $i \geq 0$. Since at a fixed time $t \in (0, T)$ the total rate of infection of the susceptible individuals by the infected ones is $ \lambda S(t) I(t)$ and since the rate at which recovered individuals spread to infected ones is $ R(t) I(t)$, we get \begin{equation}
\label{eq:transition1}\Pr {S(T_{i+1})=S(T_i)-1 \, \big| \, T_{i}<T, S(T_i), R(T_i)}= \frac{\lambda \cdot S(T_i) I(T_i)}{\lambda \cdot S(T_i)I(T_i) + R(T_i) I(T_i)}=\frac{\lambda \cdot S(T_i)}{\lambda \cdot S(T_i) + R(T_i)}
\end{equation}
and similarly
\begin{equation}
\label{eq:t2}\Pr {R(T_{i+1})=R(T_i)+1 \, \big| \,T_{i}<T_{0}, S(T_i), R(T_i)}=\frac{R(T_i)}{\lambda \cdot S(T_i) + R(T_i)}.
\end{equation}
In addition, by construction note that 
\begin{eqnarray*}
\Pr { \mathcal{S}_{N} ( \mathcal{T} _{i+1})=\mathcal{S}_{N}(\mathcal{T}_{i})-1 \, \big| \, \mathcal{T}_{i}<\mathcal{T}, \mathcal{S}_{N}(\mathcal{T}_{i}), \mathcal{R} _{\mathcal{T}_{i}}} &=& \frac{\lambda \cdot \mathcal{S}_{N}(\mathcal{T}_{i})}{\lambda \cdot \mathcal{S}_{N}(\mathcal{T}_{i}) + \mathcal{R}_{\mathcal{T}_{i}}} \\
&=&1- \Pr {\mathcal{R}(\mathcal{T}_{i+1})=\mathcal{R}_{\mathcal{T}_{i}}+1 \, \big| \,\mathcal{T}_{i}<\mathcal{T}_{0}, \mathcal{S}_{N}(\mathcal{T}_{i}), \mathcal{R}_{\mathcal{T}_{i}}}.
\end{eqnarray*}
Since $(S(0),R(0))= ( \mathcal{S}_{N}(0), \mathcal{R}(0))=(N,1)$, Theorem \ref {thm:coupling} immediately follows.\end {proof}

Observe that it is the particular form of the transition probabilities \eqref {eq:transition1} and \eqref {eq:t2} (obtained by the disappearance of $I(T_i)$ in \eqref {eq:transition1}) that has allowed to decouple the infections and the spreading of recovered individuals.

\bigskip

Thanks to Theorem \ref {thm:coupling}, we may and will  replace the chase-escape process $(S(T_i), R(T_i))_{0 \leq i \leq  \zeta}$ by the process  $( \mathcal{S}_{N}( \mathcal{T} _{i}),  \mathcal{R} ( \mathcal{T} _{i}))_{0 \leq i \leq  \xi}$. In particular, with a slight abuse, we shall say that the jump times of $ \mathcal{S}_{N}$ are the times when a susceptible individual is infected and that the jump times of $ \mathcal{R}$ are the times when a recovered individual spreads to an infected one. 
\begin{center}\emph{We shall also say that $ \mathcal{S}_{N}(t)$ (resp. $\mathcal{R}(t)$) is the number of susceptible (resp. recovered) individuals at time $t$.}\end{center}

\bigskip

Before stating several useful features of this coupling, we need to introduce some notation. Set ${\sigma}_{N}(0)=0$ and for every $1 \leq i \leq N$, let ${\sigma}_{N}(i)$ be time of the $i$-th jump time of $ \mathcal{S}_{N}$. Note that the random variables $({\sigma}_{N}(i+1)-{\sigma}_{N}(i), 0 \leq i \leq N-1)$ are independent and ${\sigma}_{N}(i+1)-{\sigma}_{N}(i)$ is distributed according to $\textnormal{\textsf{Exp}}( \lambda (N-i))$ for $0 \leq i \leq N-1$.  Observe that $ \mathcal{S}_{N}(t)=0$ for $t \geq {\sigma}_{N}(N)$.
Set also $\rho(0)=0$ and for every $i \geq 1$, let $\rho(i)$ be time of the $i$-th jump of $ \mathcal{R}$ , and note that the random variables $(\rho(i)-\rho(i-1), i \geq 1)$ are independent and that $\rho(i)-\rho(i-1)$ has the same distribution as $\textnormal{\textsf{Exp}}(i)$. A moment's thought shows that
$$ \mathcal{T}={\sigma}_{N}(N) \wedge  \rho(\min\{i \geq 1; \rho(i)< {\sigma}_{N}(i)\}).$$ 
Indeed, if $1 \leq  i<N$ and $\rho(i)<{\sigma}_{N}(i)$, since $i$ susceptible vertices have been infected for the first time at time ${\sigma}_{N}(i)$, while  $i$  infected vertices have recovered at time $\rho(i)$, this means that there are no more infected individuals at time $\rho(i)$.  Finally, let $ \overline{ \mathcal{{S}}}_{N}(t)=N- \mathcal{S}_{N}(t)$ be the number of jumps of $ \mathcal{S}_{N}$ on the interval $[0,t]$, for $0 \leq t \leq {\sigma}_{N}(N)$, and note that $ \mathcal{T}$ is the smallest of the two quantities ${\sigma}_{N}(N)$ and $\inf \{ 0 \leq t \leq {\sigma}_{N}(N) ; \overline{\mathcal{S}}_{N}(t) <  \mathcal{R}(t)-1 \} $ (with the convention $ \inf \emptyset= \infty$). See Fig.~\ref{fig:1} for an illustration.

\begin{figure*}[h]
\begin{center}
\includegraphics[scale=1]{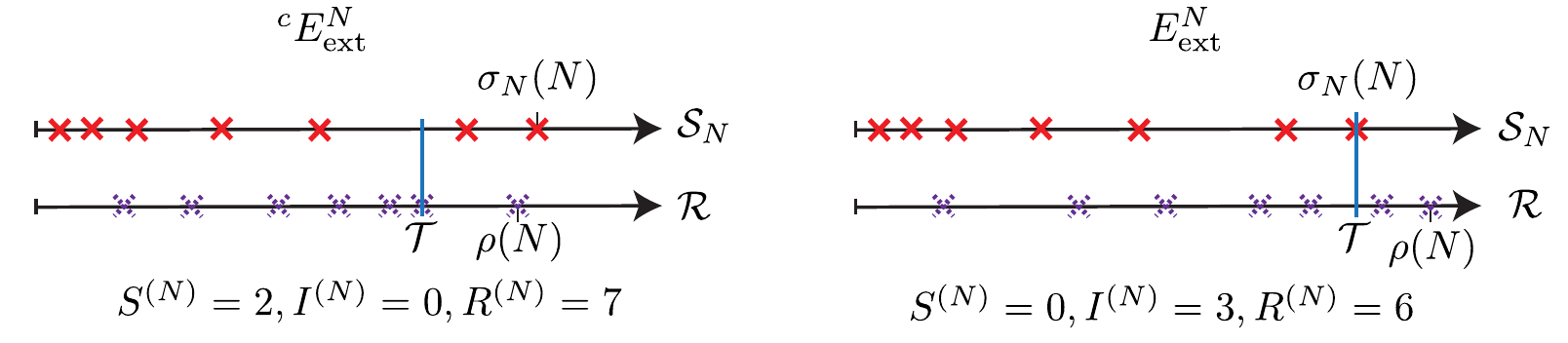}
\caption{\label{fig:1}Illustration of the two events $E^{N}_{ext}$ and ${}^c E^{N}_{ext}$ for $N=7$. The  bold crosses represent the jump times of $ \mathcal{S}_{N}$ and the  dashed crosses represent the jump times of $ \mathcal{R}$.}
\end{center}
\end{figure*}

Recall that $ E^{N}_{ext}$ is the event that there exists a (random) time when no susceptible individuals remain and that ${}^c E^{N}_{ext}$ is the complementary event. We list some  important features of our coupling which we will  extensively use in the sequel:

\begin {prop}\label {prop:ext}~
\begin{enumerate}
\item[(i)] The events $E^{N}_{ext}$ and $ \{ \overline{\mathcal{S}}_{N}(t) \geq   \mathcal{R}(t)-1 \textrm { for every } 0 \leq t \leq {\sigma}_{N}(N)\} $ are the same. In particular,
$$ \Pr{E^{N}_{ext}}= \Pr { \overline{\mathcal{S}}_{N}(t) \geq   \mathcal{R}(t)-1 \textrm { for every } 0 \leq t \leq {\sigma}_{N}(N)}.$$
\item[(ii)]On the event $E^{N}_{ext}$, the last susceptible individual is infected at time ${\sigma}_{N}(N)$. In addition, on the event $E^{N}_{ext}$,  we have $S^{(N)}=0$, $I^{(N)}=N+2- \mathcal{R}({\sigma}_{N}(N))$ and $R^{(N)}= \mathcal{R}({\sigma}_{N}(N))$.
\item[(iii)] On the event ${}^{c}E^{N}_{ext}$, the last infected individual recovers at time $ \mathcal{T}=\inf\{ t \geq 0;  \overline{\mathcal{S}}_{N}(t) <   \mathcal{R}(t)-1\}$. In addition, on the event ${}^{c}E^{N}_{ext}$, we have $S^{(N)}= \mathcal{S}_{N}( \mathcal{T})=N+2-\mathcal{R}( \mathcal{T})$, $I^{(N)}=0$, $R^{(N)}=\mathcal{R}( \mathcal{T})=N+2-\mathcal{S}( \mathcal{T})$.
 In particular, if  we have $ \mathcal{T}>t$, then $ \mathcal{S} _{N}(\rho(N)) \leq  S^{(N)} \leq  \mathcal{S} _{N}(t) $.
 \item[(iv)] We always have $R^{(N)} \leq  \mathcal{R}({\sigma}_{N}(N))$.
 \end{enumerate}
\end {prop}

In the following section, we introduce some background on Yule processes, which we will see to be very closely related to both the evolution of $\mathcal{S}_{N}$ and of $ \mathcal{R}$.

\subsection{Background on Yule processes}
\label {sec:Yule}

A Yule process of parameter $ \lambda>0$ describes the growth of a population in which each individual dies after an exponential time of parameter $ \lambda$ by giving birth to two offspring,  independently, starting from one individual.  In this section, $ (Y_{t})_{t \geq 0}$ is a Yule process of fixed parameter $ \lambda>0$.
 We now state some useful well-known results (see e.g. \cite[Section 2.5]{Nor98} and \cite[Theorem 1 in Section III.7]{AN72} for proofs).
 
\begin{prop}\label{prop:Yule}~
\begin{enumerate}
\item[(i)] Set $J_{0}=0$ and let $(J_{i})_{i \geq 1}$ be the increasing birth times of $ Y$. The random variables $(J_{i}-J_{i-1}; i \geq 1)$ are independent and $J_{i}-J_{i-1}$ has the same distribution as $\textnormal{\textsf{Exp}}( \lambda i)$.
 \item[(ii)] For every $t \geq 0$ and $k \geq 1$ we have $ \P(Y_{t}=k)= e^{- \lambda t} (1- e^{- \lambda t})^{k-1}$.
 \item[(iii)] The following convergence holds almost surely:
 $$ e^{- \lambda t} Y_{t}  \quad\mathop{\longrightarrow}^{a.s.}_{t \rightarrow \infty} \quad \mathcal{E},$$
 where $ \mathcal{E}$ is an exponential random variable of parameter one  which we call the terminal value of $Y$.
  \end{enumerate}
\end {prop}

We write $f(t-)$ for the left limit at $t \in \R$ of a function $f : \R \rightarrow \R$, when it exists. Keeping the notation introduced in the previous section, assertion (i) entails that $(\mathcal{R}(t))_{t \geq 0}$ is a Yule process of parameter $1$ and that $(\mathcal{S}_N( ({\sigma}_{N}(N)-t)-))_{0 \leq t <  {\sigma}_{N}}$ is a Yule process of parameter $ \lambda$ stopped just before its $N$-th jump time. This explains why Yule processes play an important role in the study of our model.

\medskip

We will also need the description of the Yule process conditioned on its terminal value, due to Kendall (see e.g. \cite[Section 11 in Chapter III, Theorem 2]{AN72} or \cite[Theorem 1]{Ken66}), which will allow us to couple Yule processes with Poisson processes.

\begin{prop}\label{prop:cond}Let $ \mathcal{E}$ be the terminal value of $Y$. Then, conditionally on $ \mathcal{E}$, the process $(Y_{\frac{1}{ \lambda}\log(1+ \frac{t}{ \mathcal{E}})}-1;t \geq 0)$ is a Poisson process on $ \R_{+}$ of intensity $1$ starting from $0$. More precisely, for every $0<t_{1}< \cdots < t_{k}< \infty$ and integers $n_{i} \geq 1$, $1 \leq i \leq k$, and for every Borel subset $B \subset [0, \infty]$, 
$$ \Pr {Y _{\frac{1}{ \lambda}\log (1+ \frac{t_{i}}{ \mathcal{E}} )} =n_{i} \textrm { for every }  1 \leq i \leq k , \mathcal{E} \in B }= \Pr { \mathcal{E} \in B } \Pr { \mathcal{P}_{t_{i}}=n_{i}-1 \textrm { for every }  1 \leq i \leq k },$$
where $(\mathcal{P}_{t})_{t \geq 0}$ is a Poisson process with parameter $1$ and such that $ \mathcal{P} _{0}=0$. 
\end{prop}

In other words, if $(\mathcal{P} _t)_{t \geq 0}$ is a Poisson process with unit rate starting from $0$ and $ \mathcal{E}$ is an independent exponentially distributed random variable, then the process $(Z_t)_{t \geq 0}$ defined by $Z_t=\mathcal{P}_{ \mathcal{E}( e^{ \lambda t}-1)}+1$ is a Yule process of parameter $ \lambda$ with terminal value $ \mathcal{E}$. This result will be used by transferring calculations for Yule processes to more tractable calculations involving Poisson processes.

\bigskip

We now introduce some important notation that will be used in the sequel. Let $ \mathcal{E}$ and $ \overline{\mathcal{E}}$ be two independent exponential random variables of parameter $1$, and $(\mathcal{P}_t)_{t \geq 0}$ and $ (\overline{\mathcal{P}}_t)_{t \geq 0}$  be two independent Poisson processes with unit rate starting from $0$ (all defined on the same probability space). By convention, we set $ \mathcal{P}_{0-}= \overline{  \mathcal{P} }_{0-}=-1$.  For every $t \geq 0$, set $Z_t=\mathcal{P}_{\mathcal{E}( e^{ t}-1)}+1$ and
$ \overline{Z}_t= \overline{\mathcal{P}}_{ \overline{\mathcal{E}}( e^{ \lambda t}-1)}+1$. In particular, $Z$ (resp. $ \overline {Z}$) is a Yule process with rate $1$ (resp. with rate $\lambda$) with initial value $1$. For $n \geq 0$, let $ \tau_{n}$ (resp. $ \overline { \tau}_{n}$) be the $n$-th jump time of $\mathcal{P}$ (resp. $\overline{\mathcal{P}}$). Note that $ \tau_{n}$ and $ \overline{ \tau}_n$ have both the same distribution as the sum of $n$ i.i.d.~$ \Exp(1)$ random variables.

By Proposition \ref {prop:cond}, without loss of generality, we may and will assume that the processes $( \mathcal{S}, \mathcal{R})$ and $(Z, \overline {Z})$   are coupled as follows: \begin{equation}
\label{eq:eqlaw2}\mathcal{R}(t)=Z_t \quad (t \geq 0) , \qquad \mathcal{S}_{N}(t)= \overline{Z}_{( {\sigma}_{N}(N)-t)-}, \qquad (0 \leq t \leq  {\sigma}_{N}(N)),
\end{equation}
and, for every $0 \leq i \leq N$, 
\begin{equation}
\label{eq:eqlaw1}\rho(i) = \ln \left( 1 + \frac{ \tau_{i}}{ \mathcal{E}} \right),  \qquad {\sigma}_{N}(i)= \frac{1}{ \lambda}  \ln \left( 1 + \frac{ \overline{\tau}_{N}}{  \overline{\mathcal{E}}} \right) - \frac{1}{ \lambda}  \ln \left( 1 + \frac{ \overline{\tau}_{N-i}}{  \overline{\mathcal{E}}} \right).
\end{equation}
 See Fig.~\ref {fig:coupling} for an illustration. In particular, ${\sigma}_{N}(N)=  \lambda^{-1} \ln \left( 1 + { \overline{\tau}_{N}}/{  \overline{\mathcal{E}}} \right) $, and by Proposition \ref {prop:Yule} (iii), we have
\begin{equation}
\label{eq:cv}\lambda {\sigma}_{N}(N)- \ln(N)  \quad\mathop{\longrightarrow}^{a.s.}_{N \rightarrow \infty} \quad - \ln( \overline{\mathcal{E}}), \qquad  \rho(N)- \ln(N)  \quad\mathop{\longrightarrow}^{a.s.}_{N \rightarrow \infty} \quad - \ln( \mathcal{E}),
\end{equation}
where the convergence holds almost surely.  Note also that it is the same process $ \overline {Z}$ that appears in the definition of $ \mathcal{S}_{N}$ for different values of $N$.  This is an important feature of this coupling.

\begin{figure*}[h]
\begin{center}
\includegraphics[scale=1]{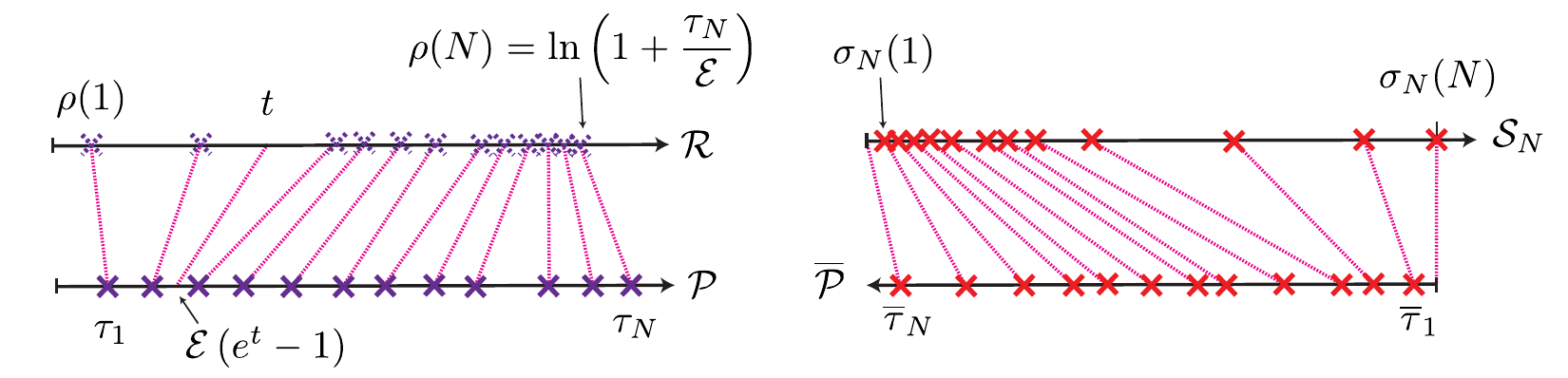}
\caption{\label{fig:coupling}Illustration of the coupling between $ \mathcal{P}$ and $ \mathcal{R}$, and between $ \overline{ \mathcal{P}}$ and $ \mathcal{S}_{N}$.}
\end{center}
\end{figure*}

 Recall that $ \overline{ \mathcal{S}}_{N}(t)=N- \mathcal{S}_{N}(t)$ is the number of jumps of $ \mathcal{S}_{N}$ on the interval $[0,t]$, for $0 \leq t \leq {\sigma}_{N}(N)$. By  \eqref {eq:eqlaw1} and \eqref {eq:eqlaw2},
 we have
 $$ \overline{ \mathcal{S}}_{N}(t)=N-1- \overline{\mathcal{P}}_{ \overline{\mathcal{E}}( e^{ \lambda( {\sigma}_{N}(N)-t)}-1)-}=N-1-\overline{\mathcal{P}}_{ (\overline{ \tau}_{N}-( \overline{\tau}_{N}+ \overline{\mathcal{E}}) \left(1- e^{- \lambda t} \right))-} \quad \textrm { for } 0 \leq  t \leq  \frac{1}{\lambda}  \ln \left(1+ \frac{\overline{ \tau}_{N}}{\overline{ \mathcal{E}}} \right) .$$
By the time-reversal property of Poisson processes,  $( \overline{\mathcal{P}}_{t}; 0 \leq t \leq \overline{\tau}_{N})$ has the same distribution as $(N-1-\overline{\mathcal{P}}_{( \overline{\tau}_{N}-t)-}; 0 \leq t \leq  \overline{\tau}_{ N})$. Indeed, conditionally on $ \overline{  \tau}_{N}$, the jump times of these processes are distributed as $N-1$ i.i.d.~points on $[0,\overline{  \tau}_{N}]$. Hence 
\begin{equation}
\label{eq:ubarre}( \overline{\mathcal{S}}_{N}(t); \quad  0 \leq t \leq  {\sigma}_{N}(N)) \qquad \mathop{=}^{(d)} \qquad  \left( \overline{\mathcal{P}}_{( \overline{\tau}_{N}+ \overline{\mathcal{E}}) \left(1- e^{- \lambda t} \right)};  \quad 0 \leq  t \leq  \frac{1}{\lambda}  \ln \left(1+ \frac{\overline{ \tau}_{N}}{\overline{ \mathcal{E}}} \right) \right).
\end{equation}
In addition, by \eqref {eq:eqlaw2}, 
 \begin{equation}
 \label{eq:v} (\mathcal{R}(t); \quad t \geq 0)=(\mathcal{P}_{\mathcal{E}( e^{ t}-1)}+1; \quad t \geq 0).
 \end{equation}

 \section{Critical value for extinction probabilities}

  The goal of this section is to provide a proof of Theorem \ref{thm:critical} by relying on the following result.
  
       \begin {lem}\label{lem:tech1} For every $ \lambda>0$ we have    $$\Pr { {\sigma}_{N}(N)<\rho(N) \textrm{ and there exists } 0 \leq t \leq {\sigma}_{N}(N) \textrm{ such that } \overline{\mathcal{S}}_{N}(t)< \mathcal{R}(t)-1}   \quad\mathop{\longrightarrow}_{N \rightarrow \infty} \quad 0.$$       
      \end{lem}  
  
  Since $  \{ \overline{\mathcal{S}}_{N}(t) \geq   \mathcal{R}(t)-1 \textrm { for every } 0 \leq t \leq {\sigma}_{N}(N)\} \subset   \{{\sigma}_{N}(N)<\rho(N)\}$, in view of Proposition \ref{prop:ext} (i), Lemma \ref{lem:tech1} entails
  $$ \Pr { {\sigma}_{N}(N)<\rho(N)} - \Pr{E^{N}_{ext}}  \quad\mathop{\longrightarrow}_{N \rightarrow \infty} \quad 0.$$
  This result comes intuitively from the fact that  up to time ${\sigma}_{N}(N)$, the jumps of $ \mathcal{S}$ are concentrated around $0$, while the jumps of $ \mathcal{R}$ are concentrated around ${\sigma}_{N}(N)$ (see Fig.~\ref{fig:2}), implying that with probability tending to one, $ \overline{\mathcal{S}}_{N}(t) \geq   \mathcal{R}(t)-1$ holds for every  $0 \leq t \leq {\sigma}_{N}(N)$ if and only if it holds for $ t={\sigma}_{N}(N)$.   However, its proof  is technical and postponed to Section \ref {sec:t1}.

\begin{figure*}[h]
\begin{center}
\includegraphics[scale=1]{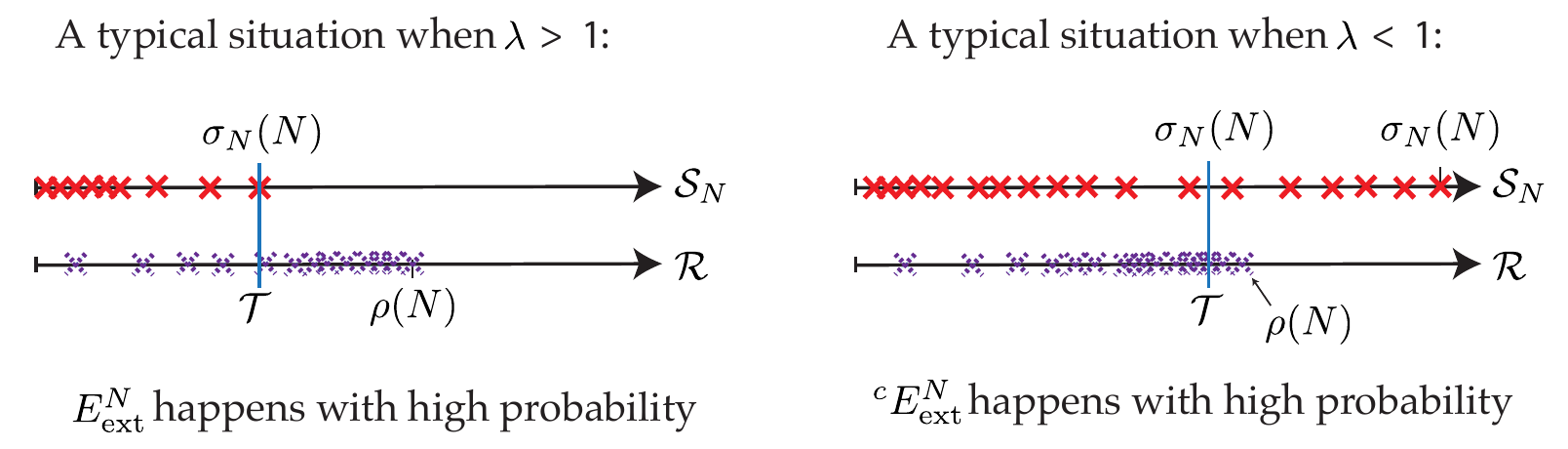}
\caption{\label{fig:2}Illustration of Lemma \ref{lem:tech1}.}
\end{center}
\end{figure*}

 \begin {proof}[Proof of Theorem \ref{thm:critical}] We have already seen that
 $ \P(E^{N}_{ext}) - \Pr {{\sigma}_{N}(N)<\rho(N)}$ converges to $0$ as $N \rightarrow \infty$. When $ \lambda \neq 1$, by \eqref {eq:cv}, ${\sigma}_{N}(N)/\rho(N)$ converges almost surely to $ 1/\lambda$, and the conclusion follows. When $ \lambda =1$, again by \eqref {eq:cv}, ${\sigma}_{N}(N)-\rho(N)$ converges almost surely to $ \ln (\mathcal{E}/ \overline{  \mathcal{E} })$ and the conclusion follows as well since $ \P(\ln ( \mathcal{E}/ \overline{ \mathcal{E}})<0)=1/2$.  
 \end {proof}

 \section{Large population limit: convergence in distribution}
 
 The goal of this section is to study the asymptotic behavior of $S^{(N)},I^{(N)}$ and $R^{(N)}$ with respect to weak convergence, and to prove in particular Theorems \ref{thm:stateS1}, \ref{thm:stateR1} and \ref{thm:stateI1}.
  
\subsection {A technical lemma}
 We first make some useful observations. If $\rho(1)<{\sigma}_{N}(1)$, then the recovered individual starts by spreading to the infected one, so that the chase-escape processes terminates with $N$ susceptible individuals and $2$ recovered individuals. Observe that  this happens with probability
\begin{equation}
\label{eq:t1} \Pr{\rho(1)<{\sigma}_{N}(1)} = \frac{1}{ \lambda N+1},
\end{equation}
and that
\begin{eqnarray}
 &&\{\rho(1)>{\sigma}_{N}(1)\} \cap \{ \overline{\mathcal{S}}_{N}(t) \geq   \mathcal{R}(t)-1 \textrm { for every } \rho(2) \leq t \leq {\sigma}_{N}(N)\} \notag\\
 && \qquad\qquad \qquad =  \quad  \{ \overline{\mathcal{S}}_{N}(t) \geq   \mathcal{R}(t)-1 \textrm { for every } 0 \leq t \leq {\sigma}_{N}(N)\} \notag \\
 && \qquad \qquad\qquad = \quad E^{N}_{ext}.\label{eq:t0}
\end{eqnarray}

We will need the following  result (see Fig.~\ref{fig:3} for an illustration).
 
     \begin {lem}\label{lem:tech2}For every $\lambda>0$, there exists a constant $C>0$ such that, for every integer $N \geq 1$,
     $$\Pr {\textrm{there exists } t \in [\rho(2), \min(\rho(N)-1/\ln(N), {\sigma}_{N}(N))] \textrm { such that } \overline{ \mathcal{S}}_{N}(t) <  \mathcal{R}(t)-1 } \leq \frac{C}{N^{2}}.$$
     \end{lem}

\begin{figure*}[h]
\begin{center}
\includegraphics[scale=1]{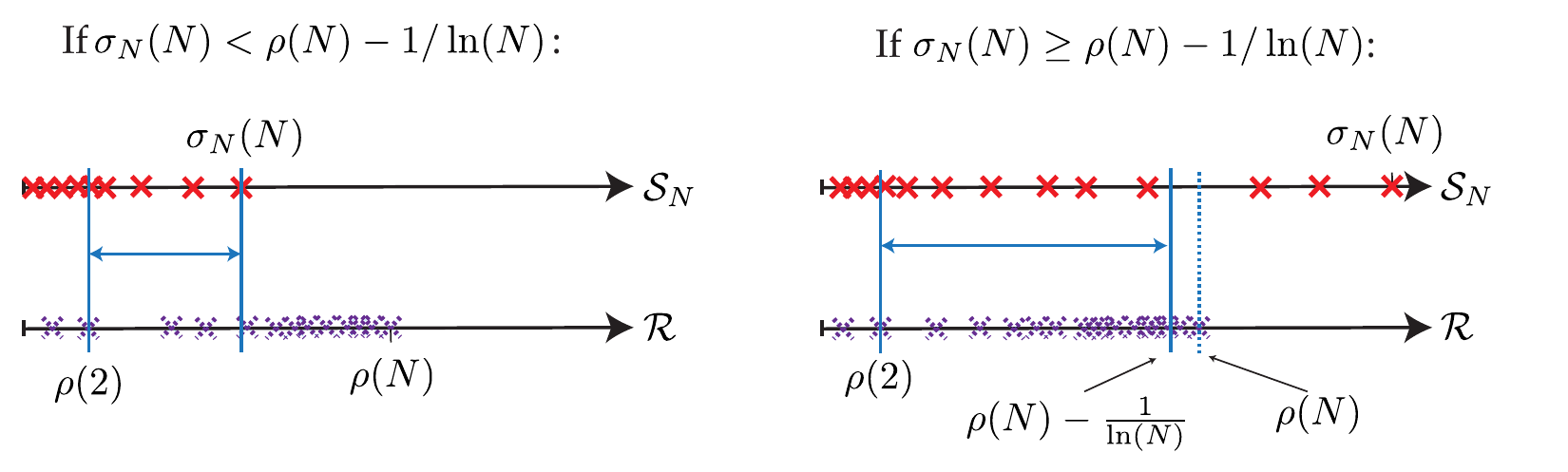}
\caption{\label{fig:3}Illustration of Lemma \ref{lem:tech2}: there exists  $t$ such that $\overline{ \mathcal{S}}_{N}(t) <  \mathcal{R}(t)-1$ in the interval delimited by the  arrows with probability less than $C/N^2$. In other words the probability that there exists  $t$  in the interval delimited by the  arrows, such that the number of dashed  crosses is greater than the number of bold crosses on $[0,t]$, is less than $C/N^2$.}
\end{center}
\end{figure*}
  
By Proposition \ref {prop:ext} (iii), this means that with probability tending to one, the process has not yet reached its absorbing state just before time $\min(\rho(N)-1/ \ln(N), {\sigma}_{N}(N))$. Notice also that
\begin{eqnarray*}
 &&\{\rho(1)>{\sigma}_{N}(1)\} \cap  \{ \overline{ \mathcal{S}}_{N}(t) \geq  \mathcal{R}(t)-1 \textrm{ for } t \in [\rho(2),\rho(N)-1/ \ln(N)]\} \notag\\
 & &\qquad\qquad \qquad =   \{\overline{ \mathcal{S}}_{N}(t) \geq  \mathcal{R}(t)-1 \textrm{ for } t \in [0,\rho(N)-1/ \ln(N)]\}
\end{eqnarray*}
Finally, we mention that for the proof of Theorem \ref{thm:stateS1}, we only need to know that $$\Pr {\textrm{there exists } t \in [\rho(2), \min(\rho(N)-1/\ln(N), {\sigma}_{N}(N))] \textrm { such that } \overline{ \mathcal{S}}_{N}(t) <  \mathcal{R}(t)-1 }  \quad\mathop{\longrightarrow}_{N \rightarrow \infty} \quad0.$$ However, we will need a bound on its speed of convergence to $0$ to establish Theorem \ref{thm:stateS2}. The additional term $1/\ln(N)$ which brings us away from ${\sigma}_{N}(N)$ plays an important role: indeed, it is not true that $\P($there exists $\rho(2) \leq t \leq {\sigma}_{N}(N)$  such that  $\overline{\mathcal{S}}_{N}(t)< \mathcal{R}(t)-1) \leq C/N^2$.

  \subsection{Susceptible vertices in the final state}
  
    In this section, we prove Theorem \ref{thm:stateS1} concerning the asymptotic behavior of $S^{(N)}$ with respect to weak convergence.

\begin{proof}[Proof of Theorem \ref{thm:stateS1}] Assume for the moment that $0 < \lambda \leq 1$. Conditionally on ${E}^{N}_{ext}$, we have $S^{(N)}=0$. It is thus enough to study the behavior of $S^{(N)}$ conditionally on ${}^{c}{E}^{N}_{ext}$.  By Lemma \ref {lem:tech1}, conditionally on ${}^{c}{E}^{N}_{ext}$, $\Pr {{\sigma}_{N}(N)>\rho(N)} \rightarrow 1$. Hence,  by  Lemma \ref{lem:tech2}, \eqref{eq:t0} and \eqref{eq:t1}, we have
$$\Pr {\textrm{there exists } t \in [0, \rho(N)-1/ \ln(N)] \textrm { such that } \overline{ \mathcal{S}}_{N}(t) <  \mathcal{R}(t)-1  | {}^{c}{E}^{N}_{ext}}  \quad\mathop{\longrightarrow}_{N \rightarrow \infty} \quad  0.$$ 
This means that conditionally on the fact that the infection does not spread to the whole graph, the time when the last infected individual recovers belongs to the interval $[\rho(N)-1/\ln(N),\rho(N)]$ with probability tending to one as $ N \rightarrow \infty$. By Proposition \ref{prop:ext} (iii), this implies that $ \mathcal{S} _{N}(\rho(N)) \leq  S^{(N)} \leq  \mathcal{S} _{N}(\rho(N)-  \ln(N)^{-1}) $, or, using \eqref{eq:eqlaw2},
\begin{equation}
\label{eq:pr} \overline{Z}_{{\sigma}_{N}(N)-\rho(N)}  \leq S^{(N)} \leq  \overline{Z}_{{\sigma}_{N}(N)-\rho(N)+\ln(N)^{-1}}.
\end{equation}
Notice that, for $u \geq 0$,\begin{equation}
\label{eq:Zu}\overline{Z}_{{\sigma}_{N}(N)-\rho(N)+u}=  \overline{\mathcal{P}}_{ \overline{\mathcal{E}}( e^{ \lambda {\sigma}_{N}(N)} e^{- \lambda \rho(N)} e^{ \lambda u}-1)}+1=\overline{\mathcal{P}}_{  (\overline{\mathcal{E}}+  \overline{  \tau}_{N})  \frac{ \mathcal{E}^{ \lambda} }{( \mathcal{E}+ \tau_{N})^{ \lambda} } e^{ \lambda u}-\overline{\mathcal{E}}}+1 \qquad (0 \leq u \leq \rho(N)),
\end{equation}
and in particular, for $u=0$, we have 
\begin{equation}
\label{eq:Zu0}\overline{Z}_{{\sigma}_{N}(N)-\rho(N)}  = \overline{\mathcal{P}}_{( \overline{ \mathcal{E}}+ \overline{ \tau}_{N}) \cdot \frac{ \mathcal{E}^ \lambda}{ ( \mathcal{E}+ \tau_{N})^ \lambda}- \overline{ \mathcal{E}}}+1.
\end{equation}  

Now assume that $ \lambda=1$.   Since $  \{{\sigma}_{N}(N)> \rho(N) \}\subset {}^{c}{E}^{N}_{ext}$, we have
\begin{eqnarray*}
\P({}^{c}{E}^{N}_{ext} \ \Delta \ \{   \mathcal{E}>\overline{ \mathcal{E}}\} ) &\leq&  \Pr{{}^{c}{E}^{N}_{ext} \ \Delta \ \{   {\sigma}_{N}(N)>\rho(N)\}}+\Pr{  \{{\sigma}_{N}(N)>\rho(N)\} \ \Delta \ \{   \mathcal{E}>\overline{ \mathcal{E}}\}} \\
&=& \Pr{{}^{c}{E}^{N}_{ext}, {\sigma}_{N}(N)<\rho(N)}+\Pr{  \{{\sigma}_{N}(N)>\rho(N)\} \ \Delta \ \{   \mathcal{E}>\overline{ \mathcal{E}}\}},
\end{eqnarray*}
where we write $ A \ \Delta \ B= ( A \cup B) \backslash A \cap B$ for two sets $A$ and $B$. By Lemma \ref {lem:tech1}, $\Pr{{}^{c}{E}^{N}_{ext}, {\sigma}_{N}(N)<\rho(N)} \rightarrow 0$ and $\Pr{  \{{\sigma}_{N}(N)>\rho(N)\} \ \Delta \ \{   \mathcal{E}>\overline{ \mathcal{E}}\}} \rightarrow 0$ as $ N \rightarrow \infty$, since ${\sigma}_{N}(N)-\rho(N)$ converges almost surely to $ \ln (\mathcal{E}/ \overline{  \mathcal{E} })$.  We conclude that  $\P({}^{c}{E}^{N}_{ext} \ \Delta \ \{   \mathcal{E}>\overline{ \mathcal{E}}\} ) \rightarrow 0$ as $ N \rightarrow \infty$. Hence it suffices to check that conditionally on $\mathcal{E}>\overline{ \mathcal{E}}$, $S^{(N)}$ converges in distribution to $G'$, where we recall that $G'$ is the random variable defined in the second remark following the statement of Theorem \ref{thm:stateS1}.

 By \eqref{eq:pr}, it is enough to check that, conditionally on $\mathcal{E}>\overline{ \mathcal{E}}$,  $ \overline{Z}_{{\sigma}_{N}(N)-\rho(N)}$ converges in distribution to $G'$ and that $  \overline{Z}_{{\sigma}_{N}(N)-\rho(N)+\ln(N)^{-1}}- \overline{Z}_{{\sigma}_{N}(N)-\rho(N)}$ converges in probability to $0$. To this end, we first show that
\begin{equation}
\label{eq:geom}\textrm { conditionally on  }  \mathcal{E}> \overline{ \mathcal{E}}, \qquad   \overline{\mathcal{P}}_{ \mathcal{E} \cdot \frac{ \overline{ \mathcal{E}}+ \overline{ \tau}_{N}}{  \mathcal{E}+ \tau_{N}}  - \overline{ \mathcal{E}}}+1  \quad\mathop{\longrightarrow}^{(d)}_{N \rightarrow \infty} \quad G'.\end{equation}
As $N \rightarrow \infty$, $ \mathcal{E} \cdot ( \overline{ \mathcal{E}}+ \overline{ \tau}_{N})/(  \mathcal{E}+ \tau_{N})  - \overline{ \mathcal{E}}$ converges almost surely to $\mathcal{E}- \overline{ \mathcal{E}}$ as $N \rightarrow \infty$. To simplify notation, we denote by $ \mathcal{L}(X)$ the law of a random variable $X$ and by $ \mathcal{L}(X|A)$ the law of $X$ conditionally on an event $A$. Hence  
$$ \mathcal{L} \left( \overline{\mathcal{P}}_{ \mathcal{E} \cdot ( \overline{ \mathcal{E}}+ \overline{ \tau}_{N})/(  \mathcal{E}+ \tau_{N})  - \overline{ \mathcal{E}}} \ \big | \mathcal{E}> \overline{ \mathcal{E}}\right)  \quad\mathop{\longrightarrow}^{(d)}_{N \rightarrow \infty} \quad  \mathcal{L} \left(  \overline{\mathcal{P}}_{\mathcal{E}- \overline{ \mathcal{E}}} \ \big| \mathcal{E}> \overline{ \mathcal{E}} \right)  .$$ 
But $ \mathcal{L}( \mathcal{E}- \overline{ \mathcal{E}} \ | \mathcal{E}> \overline{ \mathcal{E}})$   is $ \Exp(1)$, and $\overline{  \mathcal{P} }_{ \Exp(1)}+1$ has the same distribution as $G'$. This entails \eqref{eq:geom}.  We next claim that
\begin{equation}
\label{eq:claimg2}\overline{Z}_{{\sigma}_{N}(N)-\rho(N)+1/ \ln(N)}-\overline{Z}_{{\sigma}_{N}(N)-\rho(N)}  \quad\mathop{\longrightarrow}^{(\P)}_{N \rightarrow \infty} \quad 0.
\end{equation}
Since the Poisson process $ \overline{  \mathcal{P} }$ has stationary increments, from \eqref{eq:Zu} we get that
$$\overline{Z}_{{\sigma}_{N}(N)-\rho(N)+1/ \ln(N)}-\overline{Z}_{{\sigma}_{N}(N)-\rho(N)} \quad \mathop{=}^{(d)} \quad
\overline{  \mathcal{P} }_{\mathcal{E} \cdot \frac{ \overline{ \mathcal{E}}+ \overline{ \tau}_{N}}{  \mathcal{E}+ \tau_{N}}( e^{ \ln(N)^{-1} }-1)}.$$
Our claim \eqref{eq:claimg2} then follows from the fact that $\mathcal{E} \cdot \frac{ \overline{ \mathcal{E}}+ \overline{ \tau}_{N}}{  \mathcal{E}+ \tau_{N}}( e^{\ln(N)^{-1}}-1)$ converges almost surely to $0$ as $ N \rightarrow \infty$. Theorem \ref {thm:stateS1} (ii) then follows.

Now assume that $ \lambda \in (0,1)$. By the functional law of large numbers, as $ N \rightarrow \infty$, the càdlàg process $(\overline{\mathcal{P}}_{tN}/ N; t \geq 0)$ converges in probability for the Skorokhod $J_1$ topology to the deterministic process $ t \mapsto t$ (see e.g \cite[Chap. 3]{Bil99}, \cite[Chap. VI]{JS03}  for background on càdlàg processes and the Skorokhod topology). Since
$$ \frac{1}{N^{1- \lambda}} \cdot \left( ( \overline{ \mathcal{E}}+ \overline{ \tau}_{N}) \cdot \frac{ \mathcal{E}^ \lambda}{ ( \mathcal{E}+ \tau_{N})^ \lambda}- \overline{ \mathcal{E}} \right)  \quad\mathop{\longrightarrow}^{a.s.}_{N \rightarrow \infty} \quad \mathcal{E}^ \lambda,$$
it follows that $\overline{\mathcal{P}}_{( \overline{ \mathcal{E}}+ \overline{ \tau}_{N}) \cdot { \mathcal{E}^ \lambda}\cdot  ( \mathcal{E}+ \tau_{N})^ {-\lambda}- \overline{ \mathcal{E}}}/N^{1- \lambda}$ converges in probability to $ \mathcal{E}^ \lambda$ as $ N \rightarrow \infty$, meaning that $ \overline{Z}_{{\sigma}_{N}(N)-\rho(N)} / N^{1- \lambda}$ converges in probability to $ \mathcal{E}^ \lambda$ as $ N \rightarrow \infty$.

In addition, as above,
$$\overline{Z}_{{\sigma}_{N}(N)-\rho(N)+1/ \ln(N)}-\overline{Z}_{{\sigma}_{N}(N)-\rho(N)} \quad \mathop{=}^{(d)} \quad
\overline{  \mathcal{P} }_{(\overline{\mathcal{E}}+  \overline{  \tau}_{N})  \frac{ \mathcal{E}^{ \lambda} }{( \mathcal{E}+ \tau_{N})^{ \lambda} } ( e^{ \lambda \ln(N)^{-1}}-1)},$$
so that $ N^{-(1- \lambda)} \cdot  \left( \overline{Z}_{{\sigma}_{N}(N)-\rho(N)+1/ \ln(N)}-\overline{Z}_{{\sigma}_{N}(N)-\rho(N)}  \right) \rightarrow 0$ in probability as $N \rightarrow \infty$.   As above, this establishes Theorem \ref{thm:stateS1} (i).

For the third assertion, it suffices to note that  $\Pr {S^{(N)}>0}= \Pr {{}^{c}E^{N}_{ext}}$, which, by Theorem \ref {thm:critical}, tends to $0$ as $N \rightarrow \infty$ when $ \lambda>1$.
  \end{proof}

  \subsection{Recovered individuals in the final state}
  
  In this section, we prove Theorem \ref{thm:stateR1}.
    
  \begin{proof}[Proof of Theorem \ref{thm:stateR1}] First assume that  $ \lambda>1$.  Then $ \Pr {{\sigma}_{N}(N)<\rho(N)} \rightarrow 1$ as $N \rightarrow \infty$, so that by Lemma \ref {lem:tech1} and Proposition \ref {prop:ext} (i), with probability tending to one as $N \rightarrow \infty$, ${\sigma}_{N}(N)$ is the moment when the last susceptible individual is infected. At that moment, the number of recovered individuals is $ \mathcal{R}({\sigma}_{N}(N))$ by Proposition \ref{prop:ext} (ii).

It is thus sufficient establish that, for every $t \in \R$,
  \begin{equation}
  \label{eq:s2} \Es {e^{it\mathcal{R}({{\sigma}_{N}({N})})/N^{1/ \lambda}}}  \quad\mathop{\longrightarrow}_{N \rightarrow \infty} \quad  \Es { e^{ it \textnormal{\textsf{Exp}}( \overline{ \mathcal{E}}^{1/\lambda})}}.
  \end{equation}
To this end, using Lemma \ref{prop:Yule} (ii), since $ \Es{e^{it \mathcal{R}(s)}}= (1-e^s(1-e^{-it}))^{-1}$ for $s \geq 0$ and $t \in \R$, we have for every $ t \in \R$
$$   \Es {e^{it\mathcal{R}({{\sigma}_{N}(N)})/N^{1/ \lambda}}}=  \Es{\frac{1}{1-e^{{\sigma}_{N}(N)} \left( 1- e^{-i t/N^{1/ \lambda}} \right)}}.$$
   Since  $\lambda {\sigma}_{N}(N)- \ln(N)  \rightarrow - \ln( \overline{\mathcal{E}})$ a.s. as $N \rightarrow \infty$, a straightforward computation entails that
  $$\frac{1}{1-e^{{\sigma}_{N}(N)} \left( 1- e^{-i t/N^{1/ \lambda}} \right)}  \quad\mathop{\longrightarrow}^{a.s.}_{N \rightarrow \infty} \quad \frac{ \overline{\mathcal{E}}^{1/ \lambda}}{ \overline{\mathcal{E}}^{1/ \lambda}-it}= \Es{ e^{it \Exp( \overline{\mathcal{E}}^{1/ \lambda})}  \big| \overline{\mathcal{E}} }.$$
Then \eqref {eq:s2} readily follows from the dominated convergence theorem after noting that 
$$  \forall s \geq 0, \quad \forall t \in \R, \qquad \left| \frac{1}{1-e^s(1-e^{-it})}\right| \leq 1.$$
This shows (iii).

Now assume that $ \lambda=1$. By  remark (ii) following Theorem \ref {thm:stateS1}, conditionally on $ {}^{c}{E}^{N}_{ext}$, $N-R^{(N)}$ converges in distribution to a positive random variable as $N \rightarrow \infty$. Hence, conditionally on $ {}^{c}{E}^{N}_{ext}$, $R^{(N)}/N$ converges in distribution to $1$. It is thus sufficient to establish that, conditionally on ${E}^{N}_{ext}$, $R^{(N)}/N$ converges in distribution to ${2}{(1+x)^{-2}} \mathbbm {1}_{[0,1]}(x) dx$.

On the event ${E}^{N}_{ext}$, we have $ R^{(N)}=\mathcal{R}({\sigma}_{N}(N))=\mathcal{P}_{ \mathcal{E}(e^ { {\sigma}_{N}(N)}-1)}+1$. Hence it is sufficient to check that\begin{equation}
\label{eq:c1} \textrm { conditionally on  } {E}^{N}_{ext}, \qquad  \frac{\mathcal{P}_{ \mathcal{E}(e^ { {\sigma}_{N}(N)}-1)}}{N} \quad\mathop{\longrightarrow}^{(d)}_{N \rightarrow \infty} \quad  \frac{2}{(1+x)^{2}} \mathbbm {1}_{[0,1]}(x) dx.
\end{equation}
Since ${\sigma}_{N}(N)-\rho(N)$ converges almost surely to $ \ln ( \overline{ \mathcal{E}}/ \mathcal{E})$, it follows that $ \P({E}^{N}_{ext} \ \Delta \ \{  \overline{ \mathcal{E}}> \mathcal{E}\} ) \rightarrow 0$ by Lemma \ref {lem:tech1}. Therefore  \eqref {eq:c1} will follow if we establish that
\begin{equation}
\label{eq:cZ} \textrm { conditionally on  }  \overline{ \mathcal{E}}> \mathcal{E}, \qquad  \frac{\mathcal{P}_{ \mathcal{E}(e^ { {\sigma}_{N}(N)}-1)}}{N} \quad\mathop{\longrightarrow}^{(d)}_{N \rightarrow \infty} \quad \frac{2}{(1+x)^{2}} \mathbbm {1}_{[0,1]}(x) dx.
\end{equation}
Since $ \mathcal{P}_{s}$ is a Poisson random variable of parameter $s$ and since $ \mathcal{P}$ is independent of $( \overline{ \mathcal{E}}, \mathcal{E}, {\sigma}_{N}(N))$,  using the explicit formula for the characteristic function of a Poisson random variable, we get that $$ \Es {e^{i t \mathcal{P}_{ \mathcal{E}(e^ { {\sigma}_{N}(N)}-1)}/N} \big|  \overline{ \mathcal{E}}> \mathcal{E} }= \Es{ e^{  \mathcal{E}(e^ { {\sigma}_{N}(N)}-1) (e^{it/N}-1)} \big|   \overline{ \mathcal{E}}> \mathcal{E} }, \qquad t \in \R.$$
   Since  $ {\sigma}_{N}(N)- \ln(N)  \rightarrow - \ln( \overline{\mathcal{E}})$ a.s. as $N \rightarrow \infty$, a straightforward computation entails that
$$ e^{  \mathcal{E}(e^ { {\sigma}_{N}(N)}-1) (e^{it/N}-1)}  \quad\mathop{\longrightarrow}^{a.s.}_{N \rightarrow \infty} \quad e^ { \frac{ \mathcal{E}}{ \overline{ \mathcal{E}}} it}.$$
Then, by the dominated convergence theorem\begin{equation}
\label{eq:cvdom1}\Es{ e^{  \mathcal{E}(e^ { {\sigma}_{N}(N)}-1) (e^{it/N}-1)} \big|   \overline{ \mathcal{E}}> \mathcal{E} }  \quad\mathop{\longrightarrow}_{N \rightarrow \infty} \quad \Es{ e^ { \frac{ \mathcal{E}}{ \overline{ \mathcal{E}}} it}\big|   \overline{ \mathcal{E}}> \mathcal{E} }.
\end{equation}
Indeed, we have the domination
$$ \big| \exp\big(  \mathcal{E}(e^ { {\sigma}_{N}(N)}-1) (e^{it/N}-1) \big) \big| = \exp \left(  \overline{  \tau}_{N}\cdot( \cos(t/N)-1) \cdot  {\mathcal{E}}/{ \overline{ \mathcal{E}}} \right) \leq1.$$
 It then suffices to notice that the density of the random variable $\overline{ \mathcal{E}}/ \mathcal{E}$, conditionally on $ { \mathcal{E}}> \overline{ \mathcal{E}}$, is ${2}/{(1+x)^{2}}$ on $[0,1]$ and the proof of (ii) is complete.

The first assertion of Theorem \ref {thm:stateR1} easily follows from Theorem \ref {thm:stateS1} and the fact that we have $R^{(N)}+S^{(N)}=N+2$ on the event $ {}^{c}E^{N}_{ext}$ and that $ \Pr {{}^{c}E^{N}_{ext}} \rightarrow 1$ as $N \rightarrow \infty$ when $ \lambda <1$.
  \end {proof}

 \subsection {Infected individuals in the final state}

Once we have established Theorems \ref {thm:stateS1} and \ref {thm:stateR1}, the proof of Theorem \ref {thm:stateI1} is effortless. 

\begin {proof}[Proof of Theorem \ref {thm:stateI1}]  When $ \lambda \in (0,1)$, we have $ \Pr {I^{(N)}>0} = \Pr {E^{N}_{ext}} \rightarrow 0$ by Theorem \ref {thm:critical}. When $ \lambda=1$, the result follows by combining Theorem \ref {thm:stateS1} (ii) and Theorem \ref {thm:stateR1} (ii) with the equality $I^{(N)}=N+2-S^{(N)}-R^{(N)}$.  Finally, when $ \lambda>1$, the desired result is a consequence of Theorem \ref {thm:stateR1} (iii) and the fact that we have $R^{(N)}+I^{(N)}=N+2$ on the event $ E^{N}_{ext}$ and that $ \Pr {E^{N}_{ext}} \rightarrow 1$ as $N \rightarrow \infty$.
\end {proof}

\section {Final state in the large population limit:  convergence in $ \L^{p}$}
  \label {sec:LP}
  The goal of this section is to study the asymptotic behavior of $S^{(N)},I^{(N)}$ and $R^{(N)}$ with respect to $ \L^p$ convergence. 
 
 \subsection {Main results}
 
 We start by stating the results which will be proved in this section. 
 
 \begin{thm}[Number of remaining susceptible individuals]\label{thm:stateS2} The following assertions hold.
\begin{enumerate}
 \item[(i)] Assume that $ \lambda \in (0,1)$. The convergence \eqref {eq:S1} holds in $ \L^p$ for every $1 \leq p < 1/\lambda$. In particular, $ \Es { S^{(N)}} \sim N^{1- \lambda} \cdot  \Gamma( \lambda+1)$ as $ N \rightarrow \infty$.
 For $ p= 1/ \lambda$, the convergence does not hold in $ \L^p$, but we have
  $$\Es { \left(  \frac{S^{(N)}}{N^{1- \lambda}}  \right)^{1/ \lambda}}  \quad\mathop{\longrightarrow}_{N \rightarrow \infty} \quad  1+1/ \lambda.$$
 For $p> 1/ \lambda$, we have $  \Es { \left(  {S^{(N)}}/{N^{1- \lambda}}  \right)^p} \sim   N^{p \lambda-1}/{\lambda}$ as $N \rightarrow \infty$.
 \item[(ii)] Assume that $ \lambda=1$.  Then
  $ \Es {S^{(N)}}  \rightarrow  1+ \Es{G}=2$ as $N \rightarrow \infty$
  and, for $p>1$, $ \Es{ \left(S^{(N)} \right)^p } \sim N^{p-1}$ as $N \rightarrow \infty$.
  \item[(iii)] Assume that $ \lambda>1$.  Then $ \Es { \left(S^{(N)} \right) ^p}  \sim N^{p-1}/ \lambda$ for $p \geq 1$ as $N \rightarrow \infty$. In particular, $\Es {S^{(N)}}  \rightarrow 1/ \lambda$ as $N \rightarrow \infty$.
  \end{enumerate}
\end {thm}

\noindent We make several comments on these results:
\begin{enumerate}
 \item[(i)] When $ \lambda \in (0,1)$, the discrepancy concerning the $ \L^p$ convergence for $p=1/ \lambda$ happens because the event that the infected individual starts by recovering happens with probability $1/( \lambda N+1)$ (in which case the process terminates with $N$ susceptible individuals), and gives a non-negligible 
 contribution to $S^{(N)}/N^{1- \lambda}$ (in the $ \L^{p}$ sense) as soon as $p \geq  1/ \lambda$. 
\item[(ii)] When $ \lambda \geq 1$, Theorem \ref {thm:stateS2} implies that the convergences in distribution appearing in Theorem \ref {thm:stateS1} (ii) and (iii) never hold in $ \L^{p}$ when $p \geq 1$. When $ \lambda >1$, the fact that $S^{(N)}$ converges in probability to $0$ but not in $ \L^{p}$ is explained by the fact that even if $ \Pr {S^{(N)}>0}$ tends to $0$ as $N \rightarrow \infty$,  on the event $S^{(N)}>0$, $S^{(N)}$ is typically not $o(1/\Pr {S^{(N)}>0})$.
   \end{enumerate}

\begin{thm}[Number of remaining recovered individuals]\label{thm:stateR2} The following assertions hold.\begin{enumerate}
\item[(i)] Assume that $ \lambda \in (0,1)$. Then as $N \rightarrow \infty$,
 $$ N-\Es{R^{(N)}}  \quad \begin {cases} \quad \sim  \quad   \Gamma( \lambda+1) \cdot N^{1- \lambda}  & \textrm{ if } \quad  0< \lambda< \frac{\sqrt {5}-1}{2}\\
 \quad   \sim \quad \left( \frac{1}{2} \Gamma(1+1/ \lambda)+\Gamma( \lambda+1) \right)  \cdot  N^{(3- \sqrt {5})/2} & \textrm{ if } \quad \lambda=\frac{\sqrt {5}-1}{2} \\
\quad  \sim \quad   \frac{1}{2} \Gamma(1+1/ \lambda) \cdot N^{2-1/ \lambda} & \textrm{ if } \quad  \frac{\sqrt {5}-1}{2} < \lambda <1. \end{cases}$$
In particular, the convergence \eqref{eq:R1} holds in $ \L^{1}$ if and only if $ \lambda  \in (0, ( \sqrt {5}-1)/{2})$. 
\item[(ii)] Assume that $ \lambda=1$.  Then  the convergence \eqref{eq:R12} holds in $ \L^p$ for every $p \geq 1$. In particular,
$$ \Es{R^{(N)}}  \quad\mathop{\sim}_{N \rightarrow \infty} \quad  \ln(2) \cdot N.$$
\item[(iii)]
Assume that $ \lambda>1$. The convergence \eqref{eq:cvd1} holds in $ \L^p$ if and only if $1 \leq p< \lambda$. In particular,
  $$  \Es{R^{(N)}}  \quad\mathop{\sim}_{N \rightarrow \infty} \quad   \frac{1}{ \Gamma(1-1/ \lambda)} \cdot N^{1/\lambda}.$$
\end {enumerate}
\end {thm}

  \begin{thm}[Outbreak sizes]\label{thm:stateI2} The following assertions hold.\begin{enumerate}
\item[(i)] If $ \lambda \in (0,1)$, as $N \rightarrow \infty$,
 $$ \Es{I^{(N)}}  \quad \begin {cases} \quad \rightarrow  \quad   0  & \textrm{ if } \quad  0< \lambda<1/2\\
 \quad   \rightarrow \quad 1 & \textrm{ if } \quad \lambda=1/2 \\
\quad  \sim \quad  \frac{1}{2} \Gamma(1+1/ \lambda) \cdot N^{2-1/ \lambda} & \textrm{ if } \quad  1/2 < \lambda <1. \end{cases}$$
\item[(ii)] If $ \lambda=1$, then the convergence \eqref{eq:I1} holds in $ \L^{p}$ for every $p \geq 1$.  In particular, $$ \Es {I^{(N)}} \mathop{\sim}_{N \rightarrow \infty} (1- \ln(2)) \cdot N.$$
\item[(iii)] If $ \lambda>1$, the convergence \eqref{eq:I2} holds in $ \L^{p}$ if and only if $1 \leq p <  \lambda$. In particular,
$$ N-\Es {I^{(N)}}= \frac{1}{ \Gamma(1- 1/ \lambda)} \cdot N^{1/ \lambda}+o(N^{1/ \lambda}).$$
\end {enumerate}
\end {thm}

By using the relation $S^{(N)}+I^{(N)}+R^{(N)}=N+2$, we shall establish Theorem \ref{thm:stateS2} (ii) \& (iii) and Theorem \ref{thm:stateR2} (ii) \& (iii) to get Theorem  \ref{thm:stateI2} (ii) \& (iii). However, we prove  Theorem \ref{thm:stateS2} (i) and Theorem \ref{thm:stateI2} (i) to get Theorem  \ref{thm:stateR2} (i).

 \subsection {Large deviations}
 
 Here we gather several lemmas involving large deviations estimates  which will be useful later.
 
 \begin{defn}
 Let $\epsilon>0$. We say that a sequence of positive numbers $(x_n)$ is $oe_{\epsilon}(n)$ if there
exist positive constants $c,C >0$ such that $x_n \leq
C e^{-cn^{\epsilon}}$ for every $n \geq 1$, in which case we write
$x_n=oe_\epsilon(n)$. We write $x_{n}=oe(n)$ if there exists $ \epsilon>0$ such that $x_n=oe_\epsilon(n)$.
\end{defn}

\begin {lem}\label {lem:dev}The following assertions hold.
\begin{enumerate}
\item[(i)] For every $N \geq 1$ and $x \geq 1$, we have $ \Pr { \tau_{N}/N \geq x} \leq\exp(-N(x-1- \ln(x)))$, and for every $N \geq 1$ and $ x \in (0,1)$ we have $ \Pr { \tau_{N}/N \leq  x} \leq\exp(-N(x-1- \ln(x)))$
\item[(ii)] We have $\Pr { |\tau_{N}-N|>N^{3/4}}=oe(N)$.
\item[(iii)] For every $ \eta>0$, there exists a constant $C>0$ such that $ \Pr { \mathcal{P}_{(1+N^{- \eta})i}<i } \leq  \exp (- C i N^{- 2 \eta})$ for every $i,N \geq 1$.
\item[(iv)] For every $ \eta \in (0,1)$, there exists a constant $C>0$ such that $\Pr { \tau_{N} \leq N^{ \eta} \mathcal{E} } \leq \exp(-C N^{1- \eta})$ for every $N \geq 1$.
\item[(v)] For every $r \geq 1$, there exists a constant $c_{r}>0$ such that $ \Es { (1+\mathcal{P}_{s} )^{r} } \leq c_{r}(1+s^{r})$ for every $s \geq 0$.
\end{enumerate}

\end {lem}

\begin {proof}
 Since $ \overline{  \tau}_{N}$ is distributed as the sum of $N$ i.i.d.~exponential random variables of parameter $1$, we have  $  \ln (\Es { \theta \overline{  \tau}_{N}})=N \ln( 1/(1- \theta))$, so that by Markov's exponential inequality we get
$$  \forall \  N \geq 1, \forall \ x \geq 1, \qquad \Pr { \tau_{N}/N \geq x} \leq \exp \left( - N  \cdot \sup_{ \theta>0} ( \theta x- \ln(1/(1- \theta))) \right) = \exp(-N(x-1- \ln(x))).$$
One similarly shows that  $ \Pr { \tau_{N}/N \leq  x} \leq\exp(-N(x-1- \ln(x)))$ for $N \geq 1$ and $ x \in (0,1)$. The second assertion easily follows from (i). For (iii), set $ \gamma_{N}=\ln(1+N^{- \eta})$, and for $N$ sufficiently large apply Markov's exponential inequality:
$$ \Pr { \mathcal{P}_{(1+N^{- \eta})i}<i } \leq  e^{  \gamma_{N} i+ (1+N^{- \eta})i (e^{ -\gamma_{N}}-1)} = e^{-iN^{- \eta}(1- N^{ \eta} \ln( 1+ N^{- \eta}))} \leq  e^{-i N^{- 2 \eta}/4}$$
since $x \ln(1+x^{-1}) \leq 1-(4x)^{-1}$ for $x \geq 1$.  For (iv), it suffices to write, for $N$ sufficiently large,
$$ \Pr { \tau_{N} \leq N^{\eta} \mathcal{E} }= \Pr { \mathcal{E} \geq  \tau_{N}/N^{\eta} }= \Es {e^{- \tau_{N}/N^{ \eta}}}= \Es {e^{- \tau_{1}/N^{ \eta}}}^{N}= \left( 1- \frac{1}{N^{ \eta}+1} \right) ^{N} \leq e^{-N^{1- \eta}/2}.$$
Indeed, $ 1-1/(x+1) \leq e^{- 1/(2x)}$ for $x \geq 1$. Finally, for (v),  by convexity of $ x \mapsto x^{r}$ on $ \R_{+}$, it is enough to check that there exists a constant $c_{r}>0$ such that $ \Es { (\mathcal{P}_{s} )^{r} } \leq c_{r} s^{r}$ for every $s \geq 0$. To this end, first observe that $ \Pr {\mathcal{P}_{s}>x} \leq \exp(-(s-x+x \ln(x/s)))$ for $x>s$. This follows from Markov exponential's inequality $ \Pr {\mathcal{P}_{s}>x} \leq \exp(-( \gamma x- s(e^{ \gamma}-1)))$ applied with $ \gamma = \ln(x/s)$.  Then write:
\begin{eqnarray*}
\Es {(\mathcal{P}_{s} )^{r}} &=& \int_{0}^{ \infty} du  \ \Pr {\mathcal{P}_{s}>u^{1/r}}  \leq s^{r}+ \int_{s^{r}}^{ \infty} du  \ \Pr {\mathcal{P}_{s}>u^{1/r}}= s^{r}+ r \int_{s}^{ \infty} du \ \Pr {\mathcal{P}_{s}>u} u^{r-1} \\
& \leq &s^{r}+ r  \int_{s}^{ \infty} du\ e^{-(s-u+u \ln(u/s))} u^{r-1}=s^{r}+ rs^{r}  \int_{1}^{ \infty} du\ e^{-s( 1+u \ln(u)-u)} u^{r-1}.
\end{eqnarray*}
The conclusion follows since $1+u \ln(u)-u \geq 0$ for $u \geq 1$, so that $e^{-s( 1+u \ln(u)-u)}$ is decreasing in $s$ for $u \geq 1$.  This completes the proof.
 \end {proof}

\subsection {Susceptible vertices in the final state}
\label {sec:S2}

\begin {proof}[Proof of Theorem \ref{thm:stateS2}]To simplify notation, let $E^N_{1}, E^N_{2}, E^N_{3}$ be the three events defined as follows: $ E^N_{1}=  \{\rho(1)>{\sigma}_{N}(1)\}$, 
$$ E^N_{2}=  \{\overline{ \mathcal{S}}_{N}(t) \geq  \mathcal{R}(t)-1\textrm{ for } t \in [\rho(2), \min( {\sigma}_{N}(N),\rho(N)-\ln(N)^{-1})]\}, \quad E^N_{3}= \{ {\sigma}_{N}(N) \geq \rho(N)-\ln(N)^{-1}\}.$$
Notice that $E^N_{1} \cap E^N_{2} \cap {}^{c}E^N_{3} \subset E^N_{ext}$, so that $S^{(N)}=0$ on the latter event. Therefore we have
$$ \Es {F(S^{(N)})}= \Es{F(S^{(N)})\mathbbm{1}_{  {}^c E^N_{1}}}+\Es{F(S^{(N)}) \mathbbm{1}_{   E^N_{1} \cap {}^c  E^N_{2}}}+\Es{F(S^{(N)}) \mathbbm{1}_{   E^N_{1} \cap E^N_{2} \cap E^N_{3}}}$$
for every measurable function $F: \R \rightarrow \R_{+}$ such that $F(0)=0$.

Now assume that $0 < \lambda <1$ and fix $1 \leq r< 1/ \lambda$. First, by \eqref{eq:t1}, since $r< 1/ \lambda$,
$$ \Es{(S^{(N)}/N^{1- \lambda})^{r} \mathbbm{1}_{  {}^c E^N_{1}}} =  \left( \frac{N}{N^{1- \lambda}} \right)^r \cdot   \frac{1}{ \lambda N+1}  \quad\mathop{\longrightarrow}_{N \rightarrow \infty} \quad  0,$$
and by Lemma \ref{lem:tech2}
$$\Es{(S^{(N)}/N^{1- \lambda})^{r} \mathbbm{1}_{   E^N_{1} \cap {}^c E^N_{2}}} \leq  \left( \frac{N}{N^{1- \lambda}} \right)^r \cdot \frac{1}{N^2}  \quad\mathop{\longrightarrow}_{N \rightarrow \infty} \quad  0.$$
Now note that on the event $E^N_{1} \cap E^N_{2} \cap {} E^N_{3}$, our coupling implies that there are still infected  individuals remaining at time  ${\rho}(N)-\ln(N)^{-1}$. By Proposition \ref{prop:ext} (iii),  on the event $E^N_{1} \cap E^N_{2} \cap {} E^N_{3}$, we have $$ \mathcal{S} _{N}(\rho(N)) \leq  S^{(N)} \leq  \mathcal{S} _{N}(\rho(N)- \ln(N)^{-1}),$$ or, in other words,  
$\overline{Z}_{{\sigma}_{N}(N)-\rho(N)}  \leq S^{(N)} \leq  \overline{Z}_{{\sigma}_{N}(N)-\rho(N)+\ln(N)^{-1}}$. Let $E^N_{4}$ be the event defined by $E^N_{4}=  \{\tau_{N} \geq N/2, \overline{  \tau}_{N} \leq 2N, \overline{\mathcal{E}} <\sqrt{N}\}$. Since $ \Pr {{}^c E^N_{4}} =oe(N)$ by  Lemma \ref {lem:dev} (i) and $S^{(N)} \leq N$, we have $\Es{(S^{(N)}/N^{1- \lambda})^{r} \mathbbm{1}_{   {} ^c E^N_{4}}} \rightarrow 0$ as $N \rightarrow \infty$. Hence, by the previous discussion and \eqref{eq:Zu0}, for every $N$ sufficiently large so that $ \exp( \lambda \ln(N)^{-1}) \leq 2$,
 \begin{eqnarray}
 \Es{(S^{(N)}/N^{1- \lambda})^{r} \mathbbm{1}_{   E^N_{1} \cap E^N_{2} \cap E^N_{3} \cap E^N_{4}}}&\leq&  \Es { \left(  \frac{\overline{Z}_{{\sigma}_{N}(N)-\rho(N)+\ln(N)^{-1}}}{N^{1- \lambda}} \right)  ^{r}   \mathbbm{1}_{   E^N_{1} \cap E^N_{2} \cap E^N_{3} \cap E^N_{4}\}}} \notag \\
  &  \leq & \Es  { \left( \frac{\overline{\mathcal{P}}_{  ( \sqrt {N}+  {2N})  \frac{ \mathcal{E}^{ \lambda}  }{( {N/2})^{ \lambda} }  \cdot 2}+1}{N^{1- \lambda}} \right) ^{r}}. \label {eq:moche1}
 \end{eqnarray}
 Plugging the inequality appearing in Lemma \ref {lem:dev} (v) into \eqref{eq:moche1} readily implies that $  \Es {(S^{(N)}/N^{1- \lambda})^{r}} $ is bounded for every $1 \leq r < 1/ \lambda$, which shows that $S^{(N)}/N^{1- \lambda}$ converges in $ \L^p$ for every $1 \leq p < 1/ \lambda$.

When $p> 1/ \lambda$, one similarly gets that
\begin{eqnarray*}
 \Es{(S^{(N)}/N^{1- \lambda})^{p} \mathbbm{1}_{  {}^c E^N_{1}}} &=&  \left( \frac{N}{N^{1- \lambda}} \right)^p \cdot  \frac{1}{ \lambda N+1}  \quad\mathop{ \sim}_{N \rightarrow \infty} \quad  \frac{1}{\lambda} \cdot N^{p \lambda-1}, \\
 \Es{(S^{(N)}/N^{1- \lambda})^{p} \mathbbm{1}_{   E^N_{1} \cap {}^c E^N_{2}}} &\leq&\left( \frac{N}{N^{1- \lambda}} \right)^p \cdot \frac{1}{N^2} \quad\mathop{ \sim}_{N \rightarrow \infty} \quad  N^{p \lambda-2} \\
\end{eqnarray*}
Then, to simplify notation, set, for $N \geq 1$ and $u \in \R$, 
\begin{equation}
\label{eq:W}W_{N}(u)= \frac{1}{N^{1- \lambda}} \left(\overline{\mathcal{P}}_{ \frac{ \mathcal{E}^{ \lambda} (\overline{\mathcal{E}}+  \overline{  \tau}_{N})   }{( \mathcal{E}+ \tau_{N})^{ \lambda} } e^{ \lambda u}-\overline{\mathcal{E}}}+1 \right),
\end{equation}
where we set by convention $ \overline{ \mathcal{P}}_{t}=-1$ for $t<0$. As before, by Proposition \ref{prop:ext} (iii),
$$\Es{ \left(  \frac{S^{(N)}}{N^{1- \lambda}} \right) ^{p} \mathbbm{1}_{   E^N_{1} \cap E^N_{2} \cap E^N_{3}}} \leq   \Es {  W_{N}(\ln(N)^{-1}) ^{p}}.
$$
It is a simple matter to check that  $\Es { \left(  W_{N}(\ln(N)^{-1})\right) ^{p}} \rightarrow \Es { \mathcal{E}^ {\lambda p}}$. Indeed, in the proof of Theorem \ref{thm:stateS1}, we have seen that  $W_{N}(0)$ converges in probability to $ \mathcal{E}^ \lambda$. The same argument shows that $W_{N}(\ln(N)^{-1})$ also converges in probability to $ \mathcal{E}^ \lambda$, and the same argument that lead us to \eqref{eq:moche1} shows that, for every $r \geq 1$, $\Es { \left(  W_{N}(\ln(N)^{-1})\right) ^{r}}$ is bounded as $N \rightarrow \infty$. The preceding estimates give
 $$  \Es { \left(  \frac{S^{(N)}}{N^{1- \lambda}}  \right)^p} \quad\mathop{ \sim}_{N \rightarrow \infty} \quad  \frac{1}{\lambda} \cdot N^{p \lambda-1}.$$

When $p=1/ \lambda$, we similarly obtain that
\begin{eqnarray*}
 \Es{(S^{(N)}/N^{1- \lambda})^{1/\lambda} \mathbbm{1}_{  {}^c E^N_{1}}} &=&  \left( \frac{N}{N^{1- \lambda}} \right)^{1/\lambda} \cdot  \frac{1}{ \lambda N+1}    \quad\mathop{\longrightarrow}_{n \rightarrow \infty} \quad  \frac{1}{ \lambda}, \\
 \Es{(S^{(N)}/N^{1- \lambda})^{{1/\lambda}} \mathbbm{1}_{   E^N_{1} \cap {}^c E^N_{2}}} &\leq&\left( \frac{N}{N^{1- \lambda}} \right)^{1/\lambda} \cdot \frac{1}{N^2} \quad\mathop{ \sim}_{N \rightarrow \infty} \quad  \frac{1}{ N}, \\
  \limsup_{N \rightarrow \infty }\Es{(S^{(N)}/N^{1- \lambda})^{1/\lambda} \mathbbm{1}_{   E^N_{1} \cap E^N_{2} \cap E^N_{3}}} & \leq   &    \Es { \mathcal{E}}.
\end{eqnarray*}
We next claim that
\begin{equation}
\label{eq:fatou} \Es { \mathcal{E}} \leq \liminf_{N \rightarrow \infty }\Es{(S^{(N)}/N^{1- \lambda})^{1/\lambda} \mathbbm{1}_{   E^N_{1} \cap E^N_{2} \cap E^N_{3}}}.
\end{equation}
To this end, first observe that by Proposition \ref{prop:ext} (iii),
$$\Es {   W_{N}(0) ^{1/\lambda} \mathbbm{1}_{   E^N_{1} \cap E^N_{2} \cap E^N_{3}}} \leq   \Es{ \left(  \frac{S^{(N)}}{N^{1- \lambda}} \right) ^{1/\lambda} \mathbbm{1}_{   E^N_{1} \cap E^N_{2} \cap E^N_{3}}} $$
We already know that $ \Pr{E^N_{1} \cap E^N_{2} \cap E^N_{3}} \rightarrow1$ as $N \rightarrow \infty$. In addition, in the proof of Theorem \ref{thm:stateS1}, we have seen that  $W_{N}(0)$ converges in probability to $ \mathcal{E}^ \lambda$ as $N \rightarrow \infty$.  Hence the quantity $W_{N}(0) ^{1/\lambda} \mathbbm{1}_{   E^N_{1} \cap E^N_{2} \cap E^N_{3}}$ converges in probability to $ \mathcal{E}$. Our claim \eqref{eq:fatou} then follows from Fatou's Lemma and a standard extraction argument. The preceding estimates imply that $ \Es{(S^{(N)}/N^{1- \lambda})^{1/\lambda} } \rightarrow 1+1/ \lambda$ as $N \rightarrow \infty$.  This completes the proof of  (i).
 
 Now assume that $ \lambda=1$ and recall the definition of the events $E^N_{1},E^N_{2},E^N_{3}$ from the beginning of the proof and the definition of $W_{N}(u)$ from \eqref{eq:W}.  The same argument as before yields 
 $$ \Es{S^{(N)} \mathbbm{1}_{  {}^c E^N_{1}}} =  N  \cdot  \frac{1}{  N+1}    \quad\mathop{\longrightarrow}_{N \rightarrow \infty} \quad  1, \qquad \qquad \Es{S^{(N)} \mathbbm{1}_{  {}^c E^N_{2}}} \leq   N  \cdot  \frac{C}{ N^2}    \quad\mathop{\longrightarrow}_{N \rightarrow \infty} \quad  0$$
  and
  $$\Es { W_{N}(0) \mathbbm{1}_{   E^N_{1} \cap E^N_{2} \cap E^N_{3}}} \leq   \Es{   {S^{(N)}}  \mathbbm{1}_{   E^N_{1} \cap E^N_{2} \cap E^N_{3}}} \leq \Es { W_{N}(\ln(N)^{-1}) \mathbbm{1}_{   E^N_{1} \cap E^N_{2} \cap E^N_{3}}}.$$
 Since $ \lambda=1$, we have already seen that ${\sigma}_{N}(N)-\rho(N) \rightarrow \ln (\mathcal{E}/ \overline{  \mathcal{E} })$ almost surely as $N \rightarrow \infty$. Hence  $ \mathbbm{1}_{   E^N_{1} \cap E^N_{2} \cap E^N_{3}} \rightarrow \mathbbm{1}_{ \{  \mathcal{E}> \overline{ \mathcal{E}}\}}$ almost surely as $N \rightarrow \infty$. The same arguments as for the case $ \lambda \in (0,1)$ yield
  $$ \Es{   {S^{(N)}}  \mathbbm{1}_{   E^N_{1} \cap E^N_{2} \cap E^N_{3}}}  \quad\mathop{\longrightarrow}_{N \rightarrow \infty} \quad    \Es{\overline{\mathcal{P}}_{\mathcal{E}- \overline{ \mathcal{E}}} \mathbbm{1}_{ \{  \mathcal{E}> \overline{ \mathcal{E}}\}}}= 2 \Es{G'}.$$
  The preceding estimates thus entail that $ \Es{S^{(N)}} \rightarrow 1+ 2 \Es{G'}=2$ as $ N \rightarrow \infty$. If $ p>1$, one similarly obtains that
  \begin{eqnarray*}
 \Es{ \left(S^{(N)} \right)^p \mathbbm{1}_{  {}^c E^N_{1}}} &=&  N^p  \cdot  \frac{1}{  N+1}    \quad\mathop{\sim}_{N \rightarrow \infty} \quad  N^{p-1}, \\
 \Es{\left(S^{(N)} \right)^p\mathbbm{1}_{  {}^c E^N_{2}}}& \leq&   N ^p \cdot  \frac{C}{ N^2}    \quad\mathop{\sim}_{N \rightarrow \infty} \quad  C \cdot N^{p-2}, \\
 \Es{\left(S^{(N)} \right)^p \mathbbm{1}_{  E^N_{1} \cap E^N_{2} \cap E^N_{4}}}  &\displaystyle  \mathop{\longrightarrow}_{N \rightarrow \infty}& 2 \Es{(G')^p},
  \end{eqnarray*}
  implying that $ \Es{ \left(S^{(N)} \right)^p } \sim N^{p-1}$ as $N \rightarrow \infty$. This completes the proof of  (ii).
  
  For (iii), assume that $ \lambda>1$ and write  for $p \geq 1$
  $$ \Es { \left( S^{(N)} \right) ^{p}}= \frac{N^{p}}{ \lambda N+1} +\Es{ \left( S^{(N)} \right) ^{p} \mathbbm{1}_{   E^N_{1} \cap {}^c E^N_{2}} }+\Es{ \left( S^{(N)} \right) ^{p} \mathbbm{1}_{   E^N_{1} \cap E^N_{2} \cap E^N_{3}}}.$$
As before, $\Es{ \left( S^{(N)} \right) ^{p } \mathbbm{1}_{   E^N_{1} \cap {}^c E^N_{2}}} \leq N ^{p} \cdot C/N^{2} = o(N^{p-1})$ and, for $N$ sufficiently large  that $ e^{ \lambda/ \ln(N)} \leq 2$:
\begin{eqnarray*}
 \Es{    \left( {S^{(N)}} \right) ^{p} \mathbbm{1}_{   E^N_{1} \cap E^N_{2} \cap E^N_{3}}} &\leq& \Es { W_{N}(\ln(N)^{-1})^{p} \mathbbm{1}_{   E^N_{1} \cap E^N_{2} \cap E^N_{3}}} \leq  \Es{ \left( \overline{\mathcal{P}}_{ \frac{ \mathcal{E}^{ \lambda} (\overline{\mathcal{E}}+  \overline{  \tau}_{N})   }{( \mathcal{E}+ \tau_{N})^{ \lambda} } e^{ \lambda / \ln(N)}-\overline{\mathcal{E}}}+1  \right)^{p} \mathbbm{1}_{   E^N_{1} \cap E^N_{2} \cap E^N_{3}}} \\
 & \leq & c_{p} \Es {  \big(\overline{\mathcal{P}}_{ 2 { \mathcal{E}^{ \lambda} (\overline{\mathcal{E}}+  \overline{  \tau}_{N})   }/{\tau_{N}^{ \lambda} }} \big)^{p} }+ c_{p} \Pr {E^N_{3}},
\end{eqnarray*}
where we have used Lemma \ref {lem:dev} (v) for the last inequality. Since $ \lambda>1$, it is a simple matter to check using Lemma \ref {lem:dev} (ii) that the first quantity tends to $0$ as $n \rightarrow \infty$, while $\Pr{E_{3}^{N}} \rightarrow0$ since  ${\sigma}_{N}(N)/\rho(N)$ converges almost surely to $ 1/\lambda<1$. This shows that $ \Es {S^{(N)}}  \sim N^{p-1}/ \lambda$ as $N \rightarrow \infty$ and completes the proof of the Theorem. \end {proof}

\subsection {Recovered vertices in the final state}

\begin {proof}[Proof of Theorem \ref {thm:stateR2} (ii) and (iii)]
If $ \lambda=1$, note that since $R^{(N)} \leq N$, $ \Es {(R^{(N)}/N)^{p}} \leq 1$.  Since $R^{(N)}/N$ converges in distribution as $N \rightarrow \infty$, this implies that for $ \lambda=1$, $R^{(N)}/N$ converges in $ \L^{p}$ for every $p \geq 1$. 

Now assume that $ \lambda>1$. By Proposition \ref{prop:ext} (iv), $R^{(N)} \leq \mathcal{R}({\sigma}_{N}({N}))$. Since $R^{(N)}/N^{1/{ \lambda}}$ converges in distribution, it  is enough to check that for every $1 \leq r < \lambda$, $ \Es {( \mathcal{R}({\sigma}_{N}({N}))/N^{1/ \lambda})^{r}}$ is bounded as $ N \rightarrow \infty$. Using \eqref{eq:v} and recalling that $ \Es { (1+\mathcal{P}_{s} )^{r} } \leq c_{r}(1+s^{r})$, write
$$ \Es {  \left( {\mathcal{R}({\sigma}_{N}({N}))} \right) ^{r}} \leq  \Es { \left( \mathcal{P}_{ \mathcal{E} (1+ \overline{  \tau}_{N}/ \overline{  \mathcal{E} })^{1/ \lambda} } +1 \right) ^{r}} \leq c_{r} \left( 1+ \Es { \left( \mathcal{E} (1+ \overline{  \tau}_{N}/ \overline{  \mathcal{E} })^{1/ \lambda} \right) ^{r}}\right), $$
and
$$\Es {  \mathcal{E}^{r} (1+ \overline{  \tau}_{N}/ \overline{  \mathcal{E} })^{r/ \lambda} } \leq c'_{r} \left( 1+ { \Es { \overline{  \tau}_{N} ^{r/ \lambda}}} \cdot  { \Es { {\overline{  \mathcal{E} }^{-r/ \lambda}}}} \right) \leq c''_{r}  \cdot N^{r/ \lambda}.$$
Note that here we crucially need the fact that $r < \lambda$ since $ \Es{\overline{  \mathcal{E} }^{-r/ \lambda}}= \infty$ otherwise. When $p \geq  \lambda$, the convergence \eqref{eq:cvd1} cannot hold in $ \L_{p}$ since $\Es { \left( \textnormal{\textsf{Exp}}( \overline{ \mathcal{E}}^{1/\lambda} )\right) ^{p}}= \infty$.
\end {proof}

We postpone the proof of Theorem \ref {thm:stateR2} (i) since it requires Theorem \ref {thm:stateI2} (i).

\subsection {Infected vertices in the final state}

We will need the following results in the proof of Theorem \ref {thm:stateI2}.

\begin {lem}\label {lem:UV} Let $ a_{N},b_{N}$ be two sequences of positive real numbers such that $ a_{N} \sim N$ and $b_{N} \sim N$ as $N \rightarrow \infty$. If $ \lambda \in (0,1)$, the following assertions hold.
\begin{enumerate}
\item[(i)]  We have \quad $ \displaystyle \Pr {(1+ a_{N}/ \oE)^{1/ \lambda}< 1+  b_{N}/\E }  \quad\mathop{ \sim}_{N \rightarrow \infty} \quad    \frac{\Gamma \left(1 + {1}/{ \lambda} \right)}{N^{ \frac{1}{\lambda}-1}}$.
\item[(ii)] Conditionally on the event $  (1+ a_{N}/ \oE)^{1/ \lambda}< 1+  b_{N}/\E $, the random variable $1+ \frac{\E}{b_{N}}- \frac{\E}{b_{N}}  (1+ \frac{a_{N}}{\oE})^{1/ \lambda}$ converges in distribution to the uniform distribution on $[0,1]$ as $N \rightarrow \infty$. This convergence also holds in $ \L^{1}$.
\item[(iii)] We have $ \Pr { \left( 1 +\frac{b_{N}}{ \mathcal{E}}  \right) e^{-1/ \ln(N)}  \leq\left( 1 +\frac{a_{N}}{ \oE}  \right)^{ 1/\lambda} \leq  1 +\frac{b_{N}}{ \mathcal{E}}  } = o(N^{1-1/ \lambda})$.
\end{enumerate}
\end {lem}

\begin {proof} Fix $0 \leq c \leq 1$ and observe that
\begin{eqnarray*}
&&\Pr {(1+ a_{N}/ \oE)^{1/ \lambda}< 1+  (1-c)b_{N}/\E } \\
 && \qquad \qquad \qquad \qquad  =\Pr { \oE > \frac{a_{N}}{ (1+ (1-c)b_{N}/\E)^{ \lambda} -1} } = \Es { \exp \left( -\frac{a_{N}}{ (1+ (1-c)b_{N}/\E)^{ \lambda} -1}  \right) } \\
&&  \qquad \qquad \qquad \qquad =\int_{0}^{ \infty} dx \ \exp(-x) \cdot  \exp \left( -\frac{a_{N}}{ (1+ (1-c)b_{N}/x)^{ \lambda} -1}  \right) \\
&& \qquad \qquad \qquad \qquad =\frac{b_{N}(1-c)}{a_{N}^{ 1/ \lambda}} \int_{0} ^{ \infty} dx\ \exp(-x (1-c)b_N/a_{N}^{1/ \lambda}) \exp \left( - \frac{1}{(1/a_{N}^{1/ \lambda}+1/x)^{ \lambda} -1/ a_{N}} \right),
\end{eqnarray*}
where we have used a change of variables for the last equality.
It follows from the dominated convergence theorem that
$$N^{ \frac{1}{\lambda}-1} \Pr {(1+ a_{N}/ \oE)^{1/ \lambda}< 1+  (1-c)b_{N}/\E }  \quad\mathop{\longrightarrow}_{N \rightarrow \infty} \quad (1-c)\int_{0}^{ \infty} dx \ \exp(-x^{ \lambda})= (1-c)\Gamma(1+1/ \lambda).$$
Indeed, it is a simple matter to check that $( \alpha x +1)^{ \lambda}- ( \alpha x)^{ \lambda} \leq 1$ for every $ \alpha \geq 0$ and $x \geq 0$, which implies that
$$\exp \left(-  \frac{1}{(1/a_{N}^{1/ \lambda}+1/x)^{ \lambda} -1/ a_{N} } \right) \leq \exp(-x^{ \lambda}).$$
for every $x \geq 0$ and $N \geq 1$. Assertion (i) immediately follows.

For (ii), note that
$$ \Es { \left . 1+ \frac{\E}{b_{N}}- \frac{\E}{b_{N}}  \left(1+ \frac{a_{N}}{\oE} \right)^{1/ \lambda} \ \right| \ (1+ a_{N}/ \oE)^{1/ \lambda}< 1+  b_{N}/\E}= \int_{0}^{1} dc \frac{F_{N}(c)}{F_{N}(0)},$$
where 
$$F_{N}(c)=\Pr {(1+ a_{N}/ \oE)^{1/ \lambda}< 1+  (1-c)b_{N}/\E }.$$
For fixed $c \in [0,1]$, we have already seen that $F_{N}(c)/F_{N}(0) \rightarrow 1-c$, which shows that the distributional limit is uniform. In addition, since $\exp(-x (1-c)b_N/a_{N}^{1/ \lambda})  \leq 1$, we also have $F_{N}(c) \leq  b_{N} (1-c)/a_{N}^{1/ \lambda}$. Since $F_{N}(0) \sim  b_{N}/a_{N}^{1/ \lambda}$ as $N \rightarrow \infty$, there exists a constant $K>0$ such that $F_{N}(c)/F_{N}(0) \leq K(1-c)$ for every $c \in [0,1]$ and $N \geq 1$. Thus $\int_{0}^{1} dc \frac{F_{N}(c)}{F_{N}(0)} \rightarrow \int_{0}^{1} dc \ (1-c)=1/2$ as $N \rightarrow \infty$ by the dominated convergence theorem. This completes the proof of (ii).

For (iii), to simplify notation set $X_{N}=1+ \frac{\E}{b_{N}}- \frac{\E}{b_{N}}  (1+ \frac{a_{N}}{\oE})^{1/ \lambda}$ and observe that 
\begin{eqnarray*}
&&\Pr { \left( 1 +\frac{b_{N}}{ \mathcal{E}}  \right) e^{-1/ \ln(N)}  \leq\left( 1 +\frac{a_{N}}{ \oE}  \right)^{ \lambda} \leq \left( 1 +\frac{b_{N}}{ \mathcal{E}}  \right)}  \\
&& \qquad\qquad\qquad\qquad\qquad\qquad\qquad= \Pr { \left(  \frac{\E}{b_{N}} +1 \right)  \left( 1-e^{- 1/ \ln(N)}\right)   \geq X_{N} \big|  X_{N}>0} \cdot   \Pr { X_{N}>0}.
\end{eqnarray*}
By (i), $\Pr { X_{N}>0} \sim {\Gamma \left(1 + {1}/{ \lambda} \right)} \cdot {N^{ 1-\frac{1}{\lambda}}}$, and by (ii), the law of $X_{N}$, conditionally on $X_{N}>0$, converges in distribution to a uniform distribution on $[0,1]$. Since $ \left(  \frac{\E}{b_{N}} +1 \right)  \left( 1-e^{- 1/ \ln(N)}\right)$ converges almost surely to $0$ as $N \rightarrow \infty$, assertion (iii) follows. This completes the proof.
\end {proof}

\begin {cor} \label {cor:UV}
For $ \lambda \in (0,1)$ we have   $$ \displaystyle \Pr {{\sigma}_{N}(N)<\rho(N)}  \quad\mathop{ \sim}_{N \rightarrow \infty} \quad    \frac{\Gamma \left(1 + {1}/{ \lambda} \right)}{N^{ \frac{1}{\lambda}-1}}.$$
\end {cor}
\begin {proof}
Recalling that ${\sigma}_{N}(N)= \ln(1+ \overline{  \tau}_{N}/ \oE)/ \lambda$ and $\rho(N)= \ln(1+ \tau_{N}/\E)$, this is a simple consequence of Lemma  \ref {lem:UV}, since by Lemma \ref {lem:dev} (ii) we can write 
\begin{eqnarray*}
&& \Pr { \left( 1+  \frac{N+N^{3/4}}{ \oE} \right) ^{1/ \lambda}< 1+  \frac{N-N^{3/4}}{\E} } +oe(N)  \\
&& \qquad\qquad\qquad  \leq  \quad \Pr {{\sigma}_{N}(N)<\rho(N)}  \quad \leq \quad  \Pr { \left( 1+  \frac{N-N^{3/4}}{ \oE} \right) ^{1/ \lambda}< 1+  \frac{N+N^{3/4}}{\E}}+oe(N).
\end{eqnarray*}\end {proof}

In the sequel, we will often refer to Lemma \ref {lem:UV} when $a_{N}$ and $b_{N}$ are respectively replaced by $  \overline{ \tau}_{N}$ and $ \tau_{N}$. This may be justified by using the same kind of inequalities as in the proof of Corollary \ref {cor:UV} (we shall leave such details to the reader).

We will also need the following estimate.

\begin{lem}\label{lem:techn} For $ \lambda \in (0,1)$ we have $ \Pr {{\sigma}_{N}(N)<\rho(N), {\sigma}_{N}(1)>\rho(1)}= o (N^{1-1/ \lambda})$.
\end{lem}

\begin {proof}It follows from the definition of the chain $ \mathcal{S}_{N}$ that $({\sigma}_{N}(1),{\sigma}_{N}(N))$ has the same distribution as $( \mathcal{E}_{N}, \mathcal{E}_{N}+{\sigma}_{N-1}(N-1))$, where $ \mathcal{E}_{N}$ is an independent exponential random variable of parameter $N$. Let  $ \mathcal{E}_{1}$ be an independent exponential random variable of parameter $1$. Hence, since $ ( \tau_{1}, \tau_{N})$ has the same distribution as $( \mathcal{E}_{1}, \mathcal{E}_{1} + \tau_{N-1})$,
\begin{eqnarray*}
\Pr {{\sigma}_{N}(N)<\rho(N), {\sigma}_{N}(1)>\rho(1)} & \leq & \Pr{{\sigma}_{N-1}(N-1) < \rho(N), \mathcal{E}_{N} > \rho(1)} \\
& \leq & \Pr{  \ln \left(1+  \frac{\overline{  \tau}_{N-1}}{\oE}  \right) < \ln \left(1+  \frac{ \mathcal{E}_{1}+\tau_{N-1}}{\E} \right),  \mathcal{E}_{N}> \mathcal{E}_{1} }.
\end{eqnarray*}
Since $ \Pr { \mathcal{E}_{1}> \sqrt{N}}=oe(N)$, it follows that 
\begin{eqnarray*}
\Pr {{\sigma}_{N}(N)<\rho(N), {\sigma}_{N}(1)>\rho(1)}  &\leq& \Pr{  \ln \left(1+  \frac{\overline{  \tau}_{N-1}}{\oE}  \right) < \ln \left(1+  \frac{  \sqrt{N}+\tau_{N-1}}{\E} \right),  \mathcal{E}_{N}> \mathcal{E}_{1} } +oe(N) \\
&=& \frac{1}{N+1} \Pr{  \ln \left(1+  \frac{\overline{  \tau}_{N-1}}{\oE}  \right) < \ln \left(1+  \frac{  \sqrt{N}+\tau_{N-1}}{\E} \right) } +oe(N).
\end{eqnarray*}
The same argument that was used to establish Corollary \ref{cor:UV} shows that 
$$\Pr{  \ln \left(1+  \frac{\overline{  \tau}_{N-1}}{\oE}  \right) < \ln \left(1+  \frac{  \sqrt{N}+\tau_{N-1}}{\E} \right) }  \quad\mathop{\sim}_{N \rightarrow \infty} \quad   \frac{\Gamma \left(1 + {1}/{ \lambda} \right)}{N^{ \frac{1}{\lambda}-1}}.$$
The desired result follows.
\end{proof}

We are finally ready to prove Theorem \ref {thm:stateI2}.

\begin {proof}[Proof of Theorem \ref {thm:stateI2}]  When $ \lambda=1$ (resp. $ \lambda>1$), this is a simple consequence of Theorem \ref {thm:stateS2} (ii) and Theorem \ref {thm:stateR2} (ii) (resp. Theorem \ref {thm:stateS2} (iii) and Theorem \ref {thm:stateR2} (iii)) by using the equality $I^{(N)}=N+2-S^{(N)}-R^{(N)}$.

Now assume that $ 0< \lambda<1$, recall that $ E^N_{1}=  \{\rho(1)>{\sigma}_{N}(1)\}$ and let $E^{N}_{5}$ be the event
 $$E^{N}_{5}= \{ {\sigma}_{N}(N) \leq \rho(N)- \ln(N)^{-1}\}.$$
 Since $ \Pr {E^{N}_{5}}= \Pr {{\sigma}_{N}(N) \leq \rho(N)}- \Pr {\rho(N)- \ln(N)^{-1} \leq {\sigma}_{N}(N) \leq \rho(N)}$, by Corollary \ref {cor:UV} and Lemma \ref {lem:UV} (iii) we have
 $$ \Pr {E^{N}_{5}}  \quad\mathop{ \sim}_{N \rightarrow \infty} \quad    \frac{\Gamma \left(1 + {1}/{ \lambda} \right)}{N^{ \frac{1}{\lambda}-1}}.$$
Next,  since $I^{(N)}=0$ on each one of the events ${}^{c}E^{N}_{ext}$, ${}^c E^N_{1}$ and $  \{{\sigma}_{N}(N)>\rho(N)\}$, we may write
$$ \Es {I^{(N)}}= \Es {I^{(N)} \mathbbm {1}_{   \rho(N) - \ln(N)^{-1} \leq {\sigma}_{N}(N) \leq \rho(N), E^{N}_{ext}\} }}+\Es {I^{(N)} \mathbbm {1}_{   E^N_{1}, E^{N}_{5}, E^{N}_{ext}\} }}.$$
Let $A_{N}$ (resp. $B_{N}$) be the first term (resp. second term) in the last sum. We study $A_{N}$ and $B_{N}$ separately. Since $A_{N} \leq N \cdot \Pr {\rho(N) - \ln(N)^{-1} \leq {\sigma}_{N}(N) \leq \rho(N)}$, by Lemma \ref {lem:UV} we have $A_{N}=o(N^{2-1/ \lambda})$. 

We shall now establish that
$$B_{N}  \quad\mathop{\sim}_{N \rightarrow \infty} \quad \frac{\Gamma(1+1/ \lambda)}{2} \cdot {N^{2-1/ \lambda}}.$$
By Proposition \ref {prop:ext} (ii), we have $I^{(N)}=N+2- \mathcal{R}({\sigma}_{N}(N))$ on the event $E^{N}_{ext}$, so that
\begin{eqnarray*}
B_{N} &=&\Es { I^{(N)} \mathbbm {1}_{   E^{N}_{1}, E^{N}_{5}, E^{N}_{ext}\} }} \\
&=&\Es {(N+2- \mathcal{R}({\sigma}_{N}(N)))  \mathbbm {1}_{   E^{N}_{1}, E^{N}_{5}\} }}-\Es {(N+2- \mathcal{R}({\sigma}_{N}(N)) ) \mathbbm {1}_{   E^{N}_{1}, E^{N}_{5}, {}^{c}E^{N}_{ext}\} }} 
\end{eqnarray*}
By Lemma \ref {lem:tech2}, we have $\Es {(N+2- \mathcal{R}({\sigma}_{N}(N)))  \mathbbm {1}_{   E^{N}_{1},E^{N}_{5}, {}^{c}E^{N}_{ext}\} } }\leq  N \Pr { E^{N}_{1}, E^{N}_{5} , {}^{c}E^{N}_{ext}} \leq C/N$. Then write
$$\Es {(N+2- \mathcal{R}({\sigma}_{N}(N)))  \mathbbm {1}_{   E^{N}_{1}, E^{N}_{5}\} }}=\Es {(N+2- \mathcal{R}({\sigma}_{N}(N)))  \mathbbm {1}_{  E^{N}_{5}\} }}-\Es {(N+2- \mathcal{R}({\sigma}_{N}(N)))  \mathbbm {1}_{   {}^c E^{N}_{1}, E^{N}_{5}\} }}.$$
By Lemma \ref{lem:techn}, we have  $$\Es {(N+2- \mathcal{R}({\sigma}_{N}(N)))  \mathbbm {1}_{   {}^c E^{N}_{1}, E^{N}_{5}\} }} \leq (N+2) \Pr {{\sigma}_{N}(N)<\rho(N), U_{N}(1)>\rho(1)}= o (N^{2-1/ \lambda}).$$
In addition, using our coupling and the fact that $( {\mathcal{P}}_{t}; 0 \leq t \leq {\tau}_{N})$ has the same distribution as $(N-1-{\mathcal{P}}_{( \overline{\tau}_{N}-t)-}; 0 \leq t \leq  {\tau}_{ N})$, we have
\begin{eqnarray*}
\Es {(N+2- \mathcal{R}({\sigma}_{N}(N)))  \mathbbm {1}_{   E^{N}_{5}\} }} &=& 2 \Pr { E^{N}_{5}}+\Es {(N-1- \mathcal{P}_{ \E (e^{	{\sigma}_{N}(N)}-1) }  )\mathbbm {1}_{   E^{N}_{5}\} }} \\
&=&2\Pr { E^{N}_{5}}+ \Es { \mathcal{P}_{ \tau_{N}-\E (e^{	{\sigma}_{N}(N)}-1)} \mathbbm {1}_{   E^{N}_{5}\} }}.
\end{eqnarray*}
We have $2\Pr { E^{N}_{5}} = o(N^{2-1/ \lambda})$, and  $\Es { \mathcal{P}_{ \tau_{N}-\E (e^{	{\sigma}_{N}(N)}-1)} \mathbbm {1}_{   E^{N}_{5}\} }}$ can be written as
$$\Es { \mathcal{P}_{ \tau_{N}-\E (e^{	{\sigma}_{N}(N)}-1)} \mathbbm {1}_{ \{  {\sigma}_{N}(N)<\rho(N)\} }}-\Es { \mathcal{P}_{ \tau_{N}-\E (e^{	{\sigma}_{N}(N)}-1)} \mathbbm {1}_{ \{ \rho(N) - \ln(N)^{-1} \leq {\sigma}_{N}(N) \leq \rho(N)\} }}.$$
We can bound the second term of this expression by $N \cdot \Pr {\rho(N) - \ln(N)^{-1} \leq {\sigma}_{N}(N) \leq \rho(N)}$, which by Lemma \ref {lem:UV} is $o(N^{2-1/ \lambda})$. Rewrite the first expression as:
$$\Es { \mathcal{P}_{ \tau_{N}-\E (e^{	{\sigma}_{N}(N)}-1)} \mathbbm {1}_{ \{  {\sigma}_{N}(N)<\rho(N)\} }}= \Es { \mathcal{P}_{ \tau_{N} \cdot  \left( 1+ \frac{\E}{\tau_{N}}- \frac{\E}{\tau_{N}}  (1+ \frac{\overline{ \tau}_{N}}{\oE})^{1/ \lambda} \right) }  \mathbbm {1}_{ \{ 1+ \frac{\E}{\tau_{N}}- \frac{\E}{\tau_{N}}  (1+ \frac{ \overline{ \tau}_{N}}{\oE})^{1/ \lambda} \geq 0\} }}.$$
We claim that this expression is asymptotic to $\frac{\Gamma(1+1/ \lambda)}{2} \cdot {N^{2-1/ \lambda}}$ as $N \rightarrow \infty$. To prove this, as before, it is enough to establish this claim when $ \overline{  \tau}_{N}$ and $ \tau_{N}$ are respectively replaced by $a_{N}$ and $b_{N}$, where $ a_{N},b_{N}$ are two sequences of positive real numbers such that $ a_{N} \sim N$ and $b_{N} \sim N$ as $N \rightarrow \infty $ (we leave details to the reader). Recall the notation $X_{N}=1+ \frac{\E}{b_{N}}- \frac{\E}{b_{N}}  (1+ \frac{a_{N}}{\oE})^{1/ \lambda}$ and observe that
$$ \Es { \mathcal{P}_{ b_{N} X_{N} } \mathbbm {1}_{ \{ X_{N} \geq 0\} } }=  \Es {  b_{N} X_{N} \mathbbm {1}_{ \{ X_{N} \geq 0\} } }= b_{N} \cdot \Es {X_{N} | X_{N} \geq 0} \cdot  \Es {X_{N} \geq 0},$$
which, by Lemma \ref {lem:UV} (i) and (ii) is asymptotic to $\frac{\Gamma(1+1/ \lambda)}{2} \cdot {N^{2-1/ \lambda}}$ as $N \rightarrow \infty$. The preceding estimates establish that $\Es {I^{(N)}} \sim\frac{\Gamma(1+1/ \lambda)}{2} \cdot {N^{2-1/ \lambda}}$ as $N \rightarrow \infty$, and this completes the proof.
\end {proof}

Note that the factor $1/2$ is present in the asymptotic behavior of $\Es {I^{(N)}}$ and not in that of  $ \Pr {{\sigma}_{N}(N) \leq \rho(N)}$. The reason stems from the proof of Theorem \ref {thm:stateI2}: when ${\sigma}_{N}(N) \leq \rho(N)$, extinction happens roughly at time $ \tau_{N}/2$ (in the time scale of the Poisson processes), and then on average roughly $N/2$ infected vertices are present at that time.

 The proof of Theorem \ref {thm:stateR2} (i)  is now effortless:
 
 \begin {proof}[Proof of Theorem \ref {thm:stateR2} (i)]Assume that $ \lambda \in (0,1)$. By Theorem \ref {thm:stateS2} (i), we have $\Es { S^{(N)}} \sim N^{1- \lambda} \cdot  \Gamma( \lambda+1)$. It then suffices to observe that $R^{(N)}=N+2-I^{(N)}-S^{(N)}$. Indeed, by Theorem \ref {thm:stateI2} (i), when $ \lambda<(1- \sqrt{5})/2$, we have $ \Es {I^{(N)}}=o(N^{1/ \lambda})$, when $ \lambda=(1- \sqrt{5})/2$ we have $2-1/ \lambda=1- \lambda$ and $\Es {I^{(N)}} \sim \frac{1}{2} \Gamma(1+1/ \lambda) N^{1- \lambda}$ and finally when $ (1- \sqrt {5})/2 < \lambda<1$ we have $ \Es {S^{(N)}}=o(N^{2-1/ \lambda})$.
 \end {proof}

  \section{Proofs of the technical lemmas}
  
\subsection{Proof of Lemma \ref{lem:tech1}}
\label {sec:t1}
 \begin {proof} If $ \lambda \in (0,1)$, Lemma \ref{lem:tech1} follows from the fact that $ \Pr{{\sigma}_{N}(N)<\rho(N)} \rightarrow 0$ as $ N \rightarrow \infty$. Now assume that $ \lambda \geq 1$. By \eqref{eq:ubarre} and \eqref{eq:v}, it is  sufficient to show that
    $$\Pr { \ln \left( 1 + \frac{ \tau_{N}}{ \mathcal{E}} \right)> \frac{1}{ \lambda}  \ln \left( 1 + \frac{ \overline{\tau}_{N}}{  \overline{\mathcal{E}}} \right) \textrm{ and } \exists \ 0 \leq t \leq \frac{1}{ \lambda}  \ln \left( 1 + \frac{ \overline{\tau}_{N}}{  \overline{\mathcal{E}}} \right) \textrm{ such that }   \mathcal{P}_{ \mathcal{E}( e^{ t}-1)}>\overline{\mathcal{P}}_{( \overline{\tau}_{N}+ \overline{\mathcal{E}}) \left(1- e^{- \lambda t} \right)} }$$
    tends to $0$ as $N \rightarrow \infty$. By making the time change $s=( \overline{\tau}_{N}+ \overline{\mathcal{E}}) \left(1- e^{- \lambda t} \right)$, this is equivalent to showing that
     $$\Pr { \ln \left( 1 + \frac{ \tau_{N}}{ \mathcal{E}} \right)> \frac{1}{ \lambda}  \ln \left( 1 + \frac{ \overline{\tau}_{N}}{  \overline{\mathcal{E}}} \right) \textrm{ and } \exists \  0 \leq s \leq  \overline{  \tau}_{N}  \textrm{ such that }\mathcal{P}_{\mathcal{E} \left(  \left( \frac{ \overline{ \tau}_{N}+ \overline{\mathcal{E}}}{ \overline{ \tau}_{N}+ \overline{\mathcal{E}} - s}\right)^{1/ \lambda}-1\right)} >  \overline{  \mathcal{P} }_{s} }  \quad\mathop{\longrightarrow}_{N \rightarrow \infty} \quad 0 ,$$
or, again equivalently, to showing that
            $$\Pr { \ln \left( 1 + \frac{ \tau_{N}}{ \mathcal{E}} \right)> \frac{1}{ \lambda}  \ln \left( 1 + \frac{ \overline{\tau}_{N}}{  \overline{\mathcal{E}}} \right) \textrm{ and } \exists \ 1 \leq i \leq  N  \textrm{ such that }\mathcal{P}_{\mathcal{E} \left(  \left( \frac{ \overline{ \tau}_{N}+ \overline{\mathcal{E}}}{ \overline{ \tau}_{N}+ \overline{\mathcal{E}} -\overline{\tau}_{i}}\right)^{1/ \lambda}-1\right)} \geq   i }$$
    tends to $0$ as $N \rightarrow \infty$.

We first treat the case $ \lambda=1$. Since $\ln \left( 1 + \frac{ \tau_{N}}{ \mathcal{E}} \right)/\ln \left( 1 + \frac{ \overline{\tau}_{N}}{  \overline{\mathcal{E}}} \right)$ converges almost surely to $ \oE/\E$, it is sufficient to establish that
\begin{equation}
\label{eq:1}\Pr { \overline{ \mathcal{E}}> \mathcal{E} \textrm{ and } \exists \ 1 \leq i \leq  N  \textrm{ such that }\mathcal{P}_{\mathcal{E} \left( \frac{ \overline{\tau}_{i}}{ \overline{ \tau}_{N}+ \overline{ \mathcal{E}} -\overline{\tau}_{i}}\right)} \geq   i }  \quad\mathop{\longrightarrow}_{N \rightarrow \infty} \quad 0.
\end{equation}
To this end, we  separate the cases $1 \leq  i \leq \sqrt {N}$ and $i> \sqrt {N}$ and first prove that
\begin{equation}
\label{eq:00}\Pr {\exists \ 1 \leq i \leq  \sqrt {N}  \textrm{ such that }\mathcal{P}_{\mathcal{E} \left( \frac{ \overline{\tau}_{i}}{ \overline{ \tau}_{N}+ \overline{ \mathcal{E}} -\overline{\tau}_{i}}\right)} \geq   i } \quad\mathop{\longrightarrow}_{N \rightarrow \infty} \quad 0.
\end{equation}
By Lemma \ref {lem:dev},  $\Pr {  \tau_{N}< N/2}=oe(N)$ and $ \Pr { \exists \ 1 \leq i \leq  \sqrt {N}  \textrm{ such that } \overline{ \tau}_{i}>N^{3/4}}=oe(N)$, so that it is sufficient to check that  for every $M>0$, \begin{equation}
\label{eq:0}\Pr {  \exists \ 1 \leq i \leq  \sqrt{N}  \textrm{ such that }\mathcal{P}_{M { \overline{\tau}_{i}}/{ N^{3/4}}} \geq   i }   \quad\mathop{\longrightarrow}_{N \rightarrow \infty} \quad 0.
\end{equation}
Using the inequality $ \Pr { \mathcal{P}_{s} \geq i} \leq s^i /i!$, we get that
$$\Pr{\mathcal{P}_{M { \overline{\tau}_{i}}/{ N^{3/4}}} \geq   i }\leq  \frac{M^i}{i! N^{3i/4}} \Es{ \overline{\tau}^i_{i}}= (M/N^{3/4}) ^i \cdot  \frac{(2i-1)!}{i! (i-1)!} \leq (4M/N^{3/4}) ^i.$$
Hence $\Pr {  \exists \ 1 \leq i \leq  \sqrt{N}  \textrm{ such that }\mathcal{P}_{M { \overline{\tau}_{i}}/{ N}} \geq   i }   \leq  \sum_{i =1}^ \infty (4M/N^{3/4}) ^i = 4M/(N^{3/4}-4M)$, and  \eqref{eq:0} follows.

We next  show that
\begin{equation}
\label{eq:2}\Pr { \overline{ \mathcal{E}}> \mathcal{E} \textrm{ and } \exists \ \sqrt {N} \leq i \leq  N  \textrm{ such that }\mathcal{P}_{\mathcal{E} \left( \frac{ \overline{\tau}_{i}}{ \overline{ \tau}_{N}+ \overline{ \mathcal{E}} -\overline{\tau}_{i}}\right)} >  i }   \quad\mathop{\longrightarrow}_{N \rightarrow \infty} \quad 0.
\end{equation}
Since  $\mathcal{E} { \overline{\tau}_{i}}/({ \overline{ \tau}_{N}+ \overline{ \mathcal{E}} -\overline{\tau}_{i}}) \leq  \mathcal{E} \overline{\tau}_{i}/ \overline{ \mathcal{E} } $, it is sufficient to show that
$$\Pr { \overline{ \mathcal{E}}> \mathcal{E} \textrm{ and } \exists \ \sqrt {N} \leq i \leq  N  \textrm{ such that }\mathcal{P}_{\mathcal{E} \overline{\tau}_{i}/ \overline{ \mathcal{E} }} >  i } \quad\mathop{\longrightarrow}_{N \rightarrow \infty} \quad 0.
$$
Fix $ \epsilon>0$ and let $ \eta \in (0,1)$ be such that $ \Pr {\overline{ \mathcal{E}}> \mathcal{E} > (1- \eta)\overline{ \mathcal{E}} }\leq  \epsilon$. Then observe that we have $ \Pr {  \forall \ i \geq \sqrt {N}, \overline{  \tau}_{i}/i \leq 1+ \eta} \geq1- \epsilon$ for $N$ sufficiently large (this can be seen using, for example, Lemma \ref {lem:dev} (i) and the union bound). Hence it is sufficient to establish that
$$\Pr{ \exists \ \sqrt {N} \leq i \leq  N  \textrm{ such that }\mathcal{P}_{(1- \eta^{2}){i}}>  i }  \quad\mathop{\longrightarrow}_{N \rightarrow \infty} \quad 0.$$
This follows from the fact that for every $ \eta \in (0,1)$, there exists a constant $C>0$ such that we have $ \Pr {\mathcal{P}_{(1- \eta^2)i}>i} \leq  \exp(-C i)$.  Indeed, fix $ \gamma>0$ such that $ \gamma>(1- \eta^2)(e^ \gamma-1)$, and using Markov's exponential inequality write $\Pr {\mathcal{P}_{(1- \eta^2)i}>i} \leq e^{ - \gamma i} e^{(1- \eta^2) i (e^ \gamma-1)} = e^{- i( \gamma-(1- \eta^2)(e^ \gamma-1))}$. Combined with \eqref{eq:00}, this completes the proof in the case $ \lambda=1$.

We finally treat the case $ \lambda>1$.  To this end, we again separate the cases $i \leq \sqrt {N}$ and $i> \sqrt {N}$ and first note that
    \begin{equation}
    \label{eq:g1}  \Pr { \exists \  1 \leq i \leq  \sqrt {N}  \textrm{ such that }\mathcal{P}_{\mathcal{E} \left(  \left( \frac{ \overline{ \tau}_{N}+ \overline{\mathcal{E}}}{ \overline{ \tau}_{N}+ \overline{\mathcal{E}} -\overline{\tau}_{i}}\right)^{1/ \lambda}-1\right)} \geq   i } \quad\mathop{\longrightarrow}_{N \rightarrow \infty} \quad 0.
    \end{equation}
   Indeed, since $ x^{1/ \lambda} \leq x$ for $x \geq 1$, we have  
    $$  \Pr { \exists \  1 \leq i \leq  \sqrt {N}  \textrm{ such that }\mathcal{P}_{\mathcal{E} \left(  \left( \frac{ \overline{ \tau}_{N}+ \overline{\mathcal{E}}}{ \overline{ \tau}_{N}+ \overline{\mathcal{E}} -\overline{\tau}_{i}}\right)^{1/ \lambda}-1\right)} \geq   i } \leq  \Pr {\exists \ 1 \leq i \leq  \sqrt {N}  \textrm{ such that }\mathcal{P}_{\mathcal{E} \left( \frac{ \overline{\tau}_{i}}{ \overline{ \tau}_{N}+ \overline{ \mathcal{E}} -\overline{\tau}_{i}}\right)} \geq   i },$$
   which tends to $0$ as $N \rightarrow \infty$ by \eqref{eq:00}.  We next show that \begin{equation*}
 \Pr { \exists \  \sqrt {N} \leq i \leq   {N}  \textrm{ such that }\mathcal{P}_{\mathcal{E} \left(  \left( \frac{ \overline{ \tau}_{N}+ \overline{\mathcal{E}}}{ \overline{ \tau}_{N}+ \overline{\mathcal{E}} -\overline{\tau}_{i}}\right)^{1/ \lambda}-1\right)} >  i } \quad\mathop{\longrightarrow}_{N \rightarrow \infty} \quad 0.
    \end{equation*}
To this end, using the fact that $\frac{ v+ \overline{\mathcal{E}}}{ v+ \overline{\mathcal{E}} -u} \leq\frac{ u+ \overline{\mathcal{E}}}{  \overline{\mathcal{E}} } $ for every $0 \leq u \leq v$ and the fact that $ \Pr{ \overline{ \tau}_{i}< \mathcal{E} }=1/2^i$, it is sufficient to prove that
$$\Pr {  \exists \  \sqrt {N} \leq i \leq   {N}  \textrm{ such that }\mathcal{P}_{ 
\frac{2\E}{\oE^{1/ \lambda}} \overline{ \tau}_{i}^{1/ \lambda}}>i } \quad\mathop{\longrightarrow}_{N \rightarrow \infty} \quad 0.
$$

We may choose $M>0$ such that $ \Pr{ 2{ \mathcal{E} } \leq  M  \overline{  \mathcal{E}}^{1/ \lambda} \textrm{ and } \overline{ \tau}_{i} \leq i^{(1+\lambda)/2} \textrm { for } \sqrt {N} \leq i \leq N} \geq 1- \epsilon$. It is hence enough to check that 
 $$\Pr { \exists \  \sqrt {N} \leq i \leq   {N}  \textrm{ such that }\mathcal{P}_{ M i^{(1+ \lambda)/(2 \lambda)}}  >  i }  \quad\mathop{\longrightarrow}_{N \rightarrow \infty} \quad  0.$$
This  follows from the fact that since $(1+ \lambda)/(2 \lambda)<1$,  there exists a constant $C>0$ such that $ \Pr { \mathcal{P}_{ M i^{(1+ \lambda)/(2 \lambda)}}>i} \leq \exp(-Ci)$ for every $i \geq 0$ (this comes from a simple application of Markov's exponential inequality). Combined with \eqref{eq:g1}, this completes the proof in the case $ \lambda>1$.
   \end{proof}

\subsection{Proof of Lemma \ref{lem:tech2}}

\begin{proof}[Proof of Proof of Lemma \ref{lem:tech2}] Recall that here $ \lambda>0$. Set $x_{N}=1/ \ln(N)$ to simplify notation.   By \eqref{eq:ubarre} and \eqref{eq:v}, it is  sufficient to show that
    $$\Pr{ \exists \ \rho(2) \leq t \leq \rho(N)-x_N \textrm{ such that }   \overline{\mathcal{P}}_{( \overline{\tau}_{N}+ \overline{\mathcal{E}}) \left(1- e^{- \lambda t} \right)} <\mathcal{P}_{ \mathcal{E}( e^{ t}-1)}}   \leq  \frac{C}{N^2}.$$
    By making the time change $s= \mathcal{E}(e^t-1)$,      this is equivalent to showing that
   $$\Pr {\exists \ \tau_{2} \leq s \leq ( \mathcal{E}+ \tau_{N})e^{-x_N}- \mathcal{E} \textrm{ such that }  \overline{\mathcal{P}}_{( \overline{\tau}_{N}+ \overline{\mathcal{E}}) \left(1-  \frac{ \mathcal{E}^ \lambda}{ ( \mathcal{E}+s)^ \lambda} \right)} <\mathcal{P}_{s}}   \leq  \frac{C}{N^2}.$$
We first check that   \begin{equation}
\label{eq:01}\Pr {\exists \ \tau_{2} \leq s \leq  {  \tau}_{ N^{3/4}} \textrm{ such that }  \overline{\mathcal{P}}_{( \overline{\tau}_{N}+ \overline{\mathcal{E}}) \left(1-  \frac{ \mathcal{E}^ \lambda}{ ( \mathcal{E}+s)^ \lambda} \right)} <\mathcal{P}_{s}}   \leq  \frac{C}{N^2}.
\end{equation}
To this end, write
\begin{eqnarray*}
\Pr {\exists \ \tau_{2} \leq s \leq  {  \tau}_{ N^{3/4}}; \  \overline{\mathcal{P}}_{( \overline{\tau}_{N}+ \overline{\mathcal{E}}) \left(1-  \frac{ \mathcal{E}^ \lambda}{ ( \mathcal{E}+s)^ \lambda} \right)} <\mathcal{P}_{s}} &=&  \Pr { \exists \  2 \leq i \leq  N^{3/4} ; \ \overline{\mathcal{P}}_{( \overline{\tau}_{N}+ \overline{\mathcal{E}}) \left(1-  \frac{ \mathcal{E}^ \lambda}{ ( \mathcal{E}+ \tau_{i})^ \lambda} \right)} <i}\\
&\leq& \Pr { \exists \  2 \leq i \leq  N^{3/4} ; \ \overline{\mathcal{P}}_{ \overline{\tau}_{N} \left(1-  \frac{ \mathcal{E}^ \lambda}{ ( \mathcal{E}+ \tau_{i})^ \lambda} \right)} <i}
\end{eqnarray*}
Observe that $\overline{\mathcal{P}}_{ \overline{\tau}_{N} \left(1-  { \mathcal{E}^ \lambda}{ /( \mathcal{E}+ \tau_{i})^ \lambda} \right)} \geq \overline{\mathcal{P}}_{ \overline{\tau}_{N} \left(1-  1/2^\lambda \right)} $ on the event $ \tau_i \geq \mathcal{E}$, and that there exists a constant $C>0$ such that $\Pr{\overline{\mathcal{P}}_{ \overline{\tau}_{N} \left(1-  1/2^\lambda \right)}<i}\leq \exp(- CN)$ for every $N \geq 1$ and $1\leq i \leq N^{3/4}$. On the other hand, since there exists a constant $C_0>0$ such that $1- \frac{1}{(1+x)^\lambda} \geq 2 C_0 x$ for every $0 \leq x \leq 1$, it follows that on the event $ \tau_i \leq \mathcal{E}$, we have $\overline{\mathcal{P}}_{ \overline{\tau}_{N} \left(1-  { \mathcal{E}^ \lambda}{ /( \mathcal{E}+ \tau_{i})^ \lambda} \right)} \geq \overline{\mathcal{P}}_{ 2 C_0 \overline{\tau}_{N} \tau_i/ \mathcal{E}} \geq \overline{\mathcal{P}}_{ 2 C_0 \overline{\tau}_{N} \tau_i/ (\tau_i+\mathcal{E})}$. Since $ \Pr{\overline{\tau}_{N} \leq N/2}=oe(N)$, it is therefore enough to check that 
\begin{equation}
\label{eq:amq}\Pr { \exists \  2 \leq i \leq  N^{3/4} ; \ \overline{\mathcal{P}}_{   \frac{ \tau_{i} }{ \tau_{i}+\mathcal{E}}   \cdot  C_0 N   } <i } \leq  \frac{C}{N^{2}}.
\end{equation}
It is a simple matter to check that the law of $ \tau_{i}/( \tau_{i}+ \E)$ has density $ix^{i-1}$ on $[0,1]$. Hence
\begin{eqnarray*}
\Pr {\overline{\mathcal{P}}_{   \frac{\tau_{i} }{ \tau_{i}+\mathcal{E}}     \cdot  C_0 {N} } =j} &=&  \frac{1}{j!}\Es { { \left(  \frac{ \tau_{i} }{ \tau_{i}+\mathcal{E}}   \cdot  C_0 {N} \right) ^{j}} e^{-  \frac{\tau_{i} }{ \tau_{i}+\mathcal{E}}   \cdot  C_0 {N}}} \\
&=& \frac{ C_0 ^{j}}{ j! }\int_{0}^{1} dx \ i \cdot  x^{i-1} \cdot x^{j} \cdot N^{j} \cdot e^{- x  C_0 N} =  \frac{i }{j! (C_{0} N)^i} \int_{0}^{C_{0}N} u^{i+j-1} e^{-u}du \\
& \leq & \frac{i }{j! (C_{0} N)^i} \int_{0}^{\infty} u^{i+j-1} e^{-u}du = \frac{(i+j-1)!}{j!} \frac{i }{( C_0 N)^{i}}.
\end{eqnarray*}
Then, for $i \geq 2$ we have
$$ \Pr{\overline{\mathcal{P}}_{ \left( 1- \frac{ \mathcal{E} }{ \tau_{i}+\mathcal{E}} \right)  \cdot  C_0 {N}   } <i } \leq  \sum_{j=0}^{i-1}\frac{(i+j-1)!}{j!}  \cdot \frac{i }{(N C_0 )^{i}}= \frac{(2i-1)!}{(i-1)!}	\cdot \frac{1}{( C_0 N)^{i}} \leq C \left(  \frac{4i}{ C_0 e} \right) ^{i} \cdot \frac{1}{N^{i}},$$
where we have used Stirling's formula for the last inequality. Hence \eqref{eq:amq} will follow if we establish the existence of an integer $M \geq 1$ such that
$$ N^{2} \sum_{i=M}^{N^{3/4}} \left(  \frac{4i}{ C_0 e} \right) ^{i} \cdot \frac{1}{N^{i}}  \quad\mathop{\longrightarrow}_{N \rightarrow \infty} \quad 0.$$
We check that $M=9$ works by writing, for $N$ sufficiently large,
$$N^{2} \sum_{i=9}^{N^{3/4}} \left(  \frac{4i}{ C_0 e} \right) ^{i} \cdot \frac{1}{N^{i}} \leq N^{2} \sum_{i=9}^{N^{3/4}} \left(  \frac{4N^{3/4}}{ C_0 e} \right) ^{i} \cdot \frac{1}{N^{i}} \leq  N^{2} \sum_{i=9}^{ \infty} \left(  \frac{4}{ C_0 eN^{1/4}} \right) ^{i} \leq N^{2} \left(  \frac{4}{ C_0 eN^{1/4}} \right) ^{9} = C \frac{1}{N^{1/4}},$$
which tends to $0$ as $N \rightarrow \infty$.

 To complete the proof of Lemma \ref{lem:tech2} it is therefore enough to check that 
\begin{equation}
\label{eq:l1}\Pr {\exists \  { \tau}_{N^{3/4}} \leq s \leq ( \mathcal{E}+ \tau_{N})e^{-x_N}- \mathcal{E} \textrm{ such that }  \overline{\mathcal{P}}_{( \overline{\tau}_{N}+ \overline{\mathcal{E}}) \left(1-  \frac{ \mathcal{E}^ \lambda}{ ( \mathcal{E}+s)^ \lambda} \right)} <\mathcal{P}_{s}}  \leq  \frac{C}{N^2}.
\end{equation}
   A simple calculation shows that  the function $ s \mapsto \frac{1}{s}  \left(1-  \frac{ \mathcal{E}^ \lambda}{ ( \mathcal{E}+s)^ \lambda} \right)$ is decreasing on $ \R_{+}$. Hence, for $0 \leq s \leq ( \mathcal{E}+ \tau_{N})e^{-x_N}- \mathcal{E}$,
  $$  \frac{\overline{\tau}_{N}+ \overline{\mathcal{E}}}{s}  \left(1-  \frac{ \mathcal{E}^ \lambda}{ ( \mathcal{E}+s)^ \lambda} \right) \geq  \frac{ \overline{ \tau}_{N}+ \overline{ \mathcal{E}}}{ ( \mathcal{E}+ \tau_{N}) e^{-x_N}- \mathcal{E}} \left( 1- \frac{ \mathcal{E}^ \lambda}{( \mathcal{E}+ \tau_{N})^ \lambda} e^ { \lambda x_N} \right)=  \frac{ \overline{ \tau}_{N}}{ \tau_{N}} e^{x_N} \cdot A_{N},$$
 where 
 $$A_{N}= \frac{1+ \overline{ \mathcal{E}}/ \overline{ \tau}_{N} }{1- \mathcal{E}(e^{x_{N}}-1)/ \tau_{N}}\left( 1- \frac{ \mathcal{E}^ \lambda}{( \mathcal{E}+ \tau_{N})^ \lambda} e^ { \lambda x_N} \right).$$
 Hence
 \begin{eqnarray}
 &&\Pr {\exists \  { \tau}_{N^{3/4}} \leq s \leq ( \mathcal{E}+ \tau_{N})e^{-x_N}- \mathcal{E} \textrm{ such that }  \overline{\mathcal{P}}_{( \overline{\tau}_{N}+ \overline{\mathcal{E}}) \left(1-  \frac{ \mathcal{E}^ \lambda}{ ( \mathcal{E}+s)^ \lambda} \right)} <\mathcal{P}_{s}}  \notag \\
 && \qquad\qquad\qquad \qquad \leq   \Pr{\exists \  { \tau}_{ N^{3/4}} \leq s \leq ( \mathcal{E}+ \tau_{N})e^{-x_N}- \mathcal{E} \textrm{ such that }  \overline{  \mathcal{P}}_{ \frac{ \overline{ \tau}_{N}}{ \tau_{N}} e^{x_N}  A_{N}\cdot  s} < \mathcal{P}_{s}}. \label {eq:cool}
 \end{eqnarray}
 We then claim that $ \Pr { \overline{ \tau}_{N}/ \tau_{N} \leq  1-1/N^{1/4}} =oe(N)$ and that $ \Pr {A_{N} \leq 1-1/ N^{ \lambda/4}}= oe(N)$. For the first claim, by the same argument that lead us to Lemma \ref {lem:dev} (ii) we get that $ \Pr { \tau_{N} \geq  N+N^{3/5}}=oe(N)$ and $\Pr { \overline{\tau}_{N} \leq   N-N^{3/5}}=oe(N)$, so that $ \Pr {\tau_{N} \leq   N+N^{3/5}, \overline{\tau}_{N}\geq    N-N^{3/5}}=1-oe(N)$. It then suffices to notice that $ (N-N^{3/5})/(N+N^{3/5}) \geq 1-N^{-1/4}$ for $N \geq 20$. For the second one for $N$ sufficiently large we have $e^{ \lambda x_{N}} \leq 2$ and  $A_{N} \geq 1-  ( 2\mathcal{E} / ( \mathcal{E}+  \tau_{N})) ^{ \lambda}$. Hence $ \Pr {A_{N} \leq 1-1/ N^{ \lambda/4}}\leq  \Pr { \tau_{N}/ \mathcal{E} \leq 2  {N}^{1/4} }$, and the claim follows from  Lemma \ref {lem:dev} (iv). Next, since $ e^{x_{N}} \geq 1+\ln(N)^{-1}$ and  $(1-1/N^{1/4})(1-1/N^{\lambda/4})(1+ 1/\ln(N)) \geq 1+1/N^{ 1/3}$ for $N$ sufficiently large, the previous observations entail that
 $$ \Pr { \frac{ \overline{ \tau}_{N}}{ \tau_{N}} e^{x_N}  A_{N} \leq 1+1/N^{ 1/3} }=oe(N).$$
 By \eqref{eq:cool}, it is thus sufficient to show that
$$\Pr {\exists \  { \tau}_{N^{3/4}} \leq s \leq ( \mathcal{E}+ \tau_{N})e^{-x_N}- \mathcal{E} \textrm{ such that }  \overline{\mathcal{P}}_{ s (1+ N^{- 1/3})} <\mathcal{P}_{s}}  \leq  \frac{C}{N^2}.$$
Noting that  $ \Pr{ ( \mathcal{E}+ \tau_{N})e^{-{x_N}}- \mathcal{E} \geq   {  \tau}_{2N}}=oe(N)$ and that $ \Pr { \tau_{N^{3/4}} \leq N^{3/4}/2}=oe(N)$, this boils down to checking that
 $$\Pr {\exists \  {N}^{3/4}/2 \leq i \leq 2N \textrm{ such that }  \overline{\mathcal{P}}_{(1+ N^{- 1/3}) i} <i}  \leq  \frac{C}{N^2}.$$
 This easily follows from the fact that $ \Pr { \overline{\mathcal{P}}_{(1+ N^{-1/3}) i} <i} \leq e^{-Ci/ {N}^{2/3}}$ by Lemma \ref {lem:dev} (iii).
 
 \end{proof}

 \section{Extensions}

We conclude by proposing possible extensions and stating an open question. In a first direction, one may wonder what happens if instead of stopping the chase-escape process once either no infected or no susceptible individuals remain, one just looks at its final state (which is attained when no infected individuals remain). This does not change the final number of susceptible individuals, so that Theorems \ref {thm:stateS1} and \ref{thm:stateS2} remain unchanged. Limit theorems for the final number of recovered individuals then follow immediately from the relation $R^{(N)}=N+2-S^{(N)}$.

In other directions, one may imagine a possibility of immigration of susceptible vertices, or  consider a model where vertices can be of $n$ different types, and where, for $1 \leq i \leq n-1$, a vertex with type $i$ may only spread to a vertex of type $i+1$ with rate $ \lambda_{i}$. It would be interesting to study if similar limit theorems as those established in this work hold.

It also natural to study the chase-escape process on other types of graphs. In particular, what happens on the graph $ \mathbb{Z}^2$, starting with one infected vertex and a neighboring recovered vertex (and all the other vertices being susceptible)? Is it true that the critical value for $ \lambda$ is less than $1$? This question is due to James Martin and was communicated to us by Itai Benjamini.

\end {document}